\documentclass[11pt, oneside]{amsart}   	% use "amsart" instead of "article" for AMSLaTeX format
\usepackage{geometry}                		% See geometry.pdf to learn the layout options. There are lots.
\usepackage[parfill]{parskip}    		% Activate to begin paragraphs with an empty line rather than an indent
\usepackage{graphicx}
\usepackage{color}
\usepackage{amssymb,amsmath}
\newtheorem{theorem}{Theorem}
\newtheorem{lemma}[theorem]{Lemma}
\newtheorem{corollary}[theorem]{Corollary}

\theoremstyle{definition}
\newtheorem{remark}[theorem]{Remark}
\newtheorem{example}[theorem]{Example}
\newtheorem{problem}[theorem]{Problem}
\newtheorem{definition}[theorem]{Definition}
\usepackage{epstopdf}
\numberwithin{equation}{section}
\numberwithin{theorem}{section}

\author{P.D.~Dragnev}
\address{Department of Mathematical Sciences, Indiana University--Purdue University Fort Wayne, Fort Wayne, IN 46805,
          USA}
          \email{dragnevp@ipfw.edu}

\author{B.~Fuglede}
\address{Department of Mathematical Sciences, University of Copenhagen, Universitetsparken 5,
2100 Copenhagen, Denmark}
\email{fuglede@math.ku.dk}

\author{D.P.~Hardin}
\address{Center for Constructive Approximation, Department of Mathematics,
Vanderbilt University,
Nashville, TN 37240, USA}
\email{doug.hardin@vanderbilt.edu}

\author{E.B.~Saff}
\address{Center for Constructive Approximation, Department of Mathematics,
Vanderbilt University,
Nashville, TN 37240, USA}
\email{edward.b.saff@vanderbilt.edu}

\author{N.~Zorii}
\address{Institute of Mathematics,
National Academy of Sciences of Ukraine, Tereshchenkivska 3, 01601,
Kyiv-4, Ukraine}
\email{natalia.zorii@gmail.com}

\thanks{The research of the first author was supported, in part, by a Simons Foundation grant no.\ 282207. The research
of the third and the fourth authors was supported, in part, by the U.\ S.\ National Science Foundation under
grants DMS-1516400. The research of the fifth author was supported, in part, by the Scholar-in-Residence program at IPFW and by the Department of Mathematical Sciences of the University of Copenhagen.}

\begin{document}

\title[Condensers with intersecting plates]{Constrained minimum Riesz energy problems for a~condenser with~intersecting plates}

\begin{abstract}
We study the constrained minimum energy problem with an external field relative to the $\alpha$-Riesz kernel $|x-y|^{\alpha-n}$ of order $\alpha\in(0,n)$ for a generalized condenser $\mathbf A=(A_i)_{i\in I}$ in $\mathbb R^n$, $n\geqslant 3$, whose oppositely charged plates intersect each other over a set of zero capacity. Conditions sufficient for the existence of minimizers are found, and their uniqueness and vague compactness are studied. Conditions obtained are shown to be sharp. We also analyze continuity of the minimizers in the vague and strong topologies when the condenser and the constraint both vary, describe the weighted equilibrium vector potentials, and single out their characteristic properties. Our arguments are based particularly on the simultaneous use of the vague topology and a suitable semimetric structure on a set of vector measures associated with $\mathbf A$, and the establishment of completeness theorems for proper semimetric spaces. The results remain valid for the logarithmic kernel on $\mathbb R^2$ and $\mathbf A$ with compact $A_i$, $i\in I$. The study is illustrated by several examples.
\end{abstract}
\maketitle

\section{Introduction}
The purpose of the paper is to study minimum energy problems with an external field  (also known in the literature as weighted minimum energy problems) relative to the $\alpha$-Riesz kernel $\kappa_\alpha(x,y)=|x-y|^{\alpha-n}$ of order $\alpha\in(0,n)$ on $\mathbb R^n$, $n\geqslant3$, where $|x-y|$ denotes the Euclidean distance between $x,y\in\mathbb R^n$ and infimum is taken over classes of vector measures $\boldsymbol\mu=(\mu^i)_{i\in I}$ associated with a generalized condenser $\mathbf A=(A_i)_{i\in I}$.
More precisely, a finite ordered collection $\mathbf A$ of closed sets $A_i\subset\mathbb R^n$, $i\in I$, termed {\it plates\/}, with the sign $s_i=\pm1$ prescribed is a {\it generalized condenser\/} if oppositely signed plates intersect each other over a set of $\alpha$-Riesz capacity zero, while $\boldsymbol\mu=(\mu^i)_{i\in I}$ is {\it associated with\/} $\mathbf A$ if each $\mu^i$, $i\in I$, is a positive scalar Radon measure on $\mathbb R^n$ supported by $A_i$. Note that any two equally signed plates may intersect each other over a set of nonzero $\alpha$-Riesz capacity (or even coincide).  In accordance with an electrostatic interpretation of a condenser, we say that the interaction between the components $\mu^i$, $i\in I$, of $\boldsymbol\mu$ is characterized by the matrix $(s_is_j)_{i,j\in I}$, so that the $\alpha$-Riesz energy of $\boldsymbol\mu$ is defined by
\[\kappa_\alpha(\boldsymbol\mu,\boldsymbol\mu):=\sum_{i,j\in I}\,s_is_j\iint|x-y|^{\alpha-n}\,d\mu^i(x)\,d\mu^j(y).\]
The difficulties appearing in the course of our investigation are caused by the fact that a short-circuit between $A_i$ and $A_j$ with $s_is_j=-1$ may occur since those plates may have zero Euclidean distance; see Theorem~\ref{pr:nons} below providing an example of a condenser with {\it no\/} $\alpha$-Riesz energy minimizer.
Therefore it is meaningful to ask what kinds of additional requirements on the objects under consideration will prevent this blow-up effect, and secure that
a minimizer for the corresponding minimum energy problem does exist. This is shown to hold if we impose a proper upper constraint $\boldsymbol\sigma=(\sigma^i)_{i\in I}$ on the vector measures in question (see Section~\ref{sec-Gauss-c} for a formulation of the constrained problem).

Having in mind a further extension of the theory, we formulate main definitions and prove auxiliary results for a general {\it strictly positive definite\/} kernel $\kappa$ on a locally compact space $X$ (Sections~\ref{sec-pr}--\ref{sec-Gauss-c}). The approach developed for these $\kappa$ and $X$ is mainly based on the simultaneous use of the vague topology and a suitable semimetric structure on a set $\mathcal E^+_\kappa(\mathbf A)$ of all vector measures of finite energy associated with $\mathbf A$ (see Section~\ref{sec-Metric} for a definition of this semimetric structure). A key observation behind this approach is the fact that since a nonzero positive scalar measure of finite energy does not charge any set of zero capacity, there corresponds to every $\boldsymbol\mu\in\mathcal E^+_\kappa(\mathbf A)$ a scalar ({\it signed\/}) Radon measure $R\boldsymbol\mu=\sum_{i\in I}\,s_i\mu^i$ on $X$, and the mapping $R$ preserves the energy (semi)metric (Theorem~\ref{lemma:semimetric}), i.e.
\[\|\boldsymbol\mu_1-\boldsymbol\mu_2\|_{\mathcal E^+_\kappa(\mathbf A)}=\|R\boldsymbol\mu_1-R\boldsymbol\mu_2\|_{\mathcal E_\kappa(X)}.\]
Here $\mathcal E_\kappa(X)$ is the pre-Hilbert space of all scalar Radon measures on $X$ with finite energy. This implies that {\it the semimetric on\/ $\mathcal E^+_\kappa(\mathbf A)$ is a metric if and only if any two equally signed plates intersect each other only in a set of zero capacity\/}.\footnote{See Lemma~\ref{inj} and its proof providing an explanation of this phenomenon.} This approach extends that from \cite{ZPot1}--\cite{ZPot3} where the oppositely charged plates were assumed to be mutually disjoint.

Based on the convexity of the class of vector measures admissible for the problem in question, the isometry between $\mathcal E_\kappa(\mathbf A)$ and its $R$-image, and the pre-Hil\-bert structure on the space $\mathcal E_\kappa(X)$, we analyze the uniqueness of solutions (Lemma~\ref{lemma:unique:}). In view of the above observation, this solution is unique whenever any two equally signed plates intersect each other only in a set of zero capacity; otherwise any two solutions have equal $R$-images.

As for the vague topology, crucial to our arguments is Lemma~\ref{lemma-rel-cl} which asserts that if the $\sigma^i$, $i\in I$, are {\it bounded\/} then the admissible measures form a {\it vaguely compact\/} space. Intuitively this is clear since $\sigma^i(X)<\infty$ implies that $\sigma^i(U_\infty)<\varepsilon$ for any sufficiently small neighborhood $U_\infty$ of the point at infinity and $\varepsilon>0$ small enough. Thus, under the vague convergence of a net of positive scalar measures $\mu_s^i\leqslant\sigma^i$, $s\in S$, to $\mu^i$ {\it no\/} part of the total mass $\mu_s^i(X)$ can disappear at infinity.

This general approach is further specified for the $\alpha$-Riesz kernels $\kappa_\alpha$ of order $\alpha\in(0,n)$ on $\mathbb R^n$. Due to the establishment of completeness results for proper semimetric subspaces of $\mathcal E^+_{\kappa_\alpha}(\mathbf A)$ (Theorems~\ref{complete} and~\ref{cor-complete}),  we have succeeded in working out a substantive theory for the constrained $\alpha$-Riesz minimum energy problems. The theory developed includes sufficient conditions for the existence of minimizers (see Theorems~\ref{th-main-comp} and~\ref{th-main}, the latter referring to a generalized condenser whose plates may be noncompact); sufficient conditions obtained in Theorem~\ref{th-main} are shown by Theorem~\ref{th:sharp} to be sharp. Theorems~\ref{th-main-comp} and~\ref{th-main} are illustrated by Examples~\ref{ex-2} and~\ref{ex-2'}, respectively. We establish continuity of the minimizers in the vague and strong topologies when the condenser $\mathbf A$ and the constraint $\boldsymbol\sigma$ both vary (Theorem~\ref{th:cont}), and also describe the weighted vector potentials of the minimizers and specify their characteristic properties (Theorem~\ref{desc-pot}). Finally, in Section~\ref{sec:dual} we provide a duality relation between non-weighted constrained and weighted unconstrained minimum $\alpha$-Riesz problems for a capacitor ($I=\{1\}$), thereby extending the logarithmic potential result of \cite[Corollary 2.15]{DS}, now for a closed (not necessarily compact) set. For this purpose we utilize the established characteristic properties of the solutions to such extremal problems (see Theorem~\ref{desc-pot} below and \cite[Theorem~7.3]{ZPot2}).

The results obtained in Sections~\ref{sec:solvI} and \ref{sec:cont}--\ref{sec:dual} and the approach developed remain valid for the logarithmic kernel on $\mathbb R^2$ and $\mathbf A$ with compact $A_i$, $i\in I$ (compare with \cite{BC}). However, in the case where at least one of the plates of a generalized condenser is noncompact, a refined analysis is still provided as yet only for the Riesz kernels. This is caused by the fact that the above-mentioned completeness results (Theorems~\ref{complete} and~\ref{cor-complete}), crucial to our investigation, are substantially based on the earlier result of the fifth named author \cite[Theorem~1]{ZUmzh} which states that, in contrast to the fact that the pre-Hil\-bert space $\mathcal E_{\kappa_\alpha}(\mathbb R^n)$ is {\it incomplete\/} in the topology determined by the $\alpha$-Riesz energy norm \cite{Car}, the metric subspace of all $\nu\in\mathcal E_{\kappa_\alpha}(\mathbb R^n)$ such that $\nu^\pm$ are supported by fixed closed disjoint sets $F_i$, $i=1,2$, respectively, is nevertheless {\it complete\/}. In turn, the quoted theorem has been established with the aid of Deny's theorem \cite{D1} showing that $\mathcal E_{\kappa_\alpha}(\mathbb R^n)$ can be completed by making use of tempered distributions on $\mathbb R^n$ with finite energy, defined in terms of Fourier transforms, and this result by Deny seems not yet to have been extended to other classical kernels.

\begin{remark}\label{rem-intr}Regarding methods and approaches applied, assume for a moment that
\begin{equation}\label{eq-intr}\kappa|_{A_i\times A_j}\leqslant M<\infty\text{ \ whenever \ }s_is_j=-1,\end{equation}
$\kappa$ being a strictly positive definite kernel on a locally compact space $X$. (For the $\alpha$-Riesz kernels on $\mathbb R^n$, (\ref{eq-intr}) holds if and only if oppositely signed plates have nonzero Euclidean distance.) If moreover $\kappa$ is perfect \cite{F1}, then a fairly general theory of {\it un\-const\-rain\-ed\/} minimum weighted energy problems over $\boldsymbol\mu\in\mathcal E_\kappa^+(\mathbf A)$ has been developed in \cite{ZPot2,ZPot3} (see Remark~\ref{r-3} below for a short survey). The approach developed in \cite{ZPot2,ZPot3} substantially used requirement (\ref{eq-intr}), which made it possible to extend Cartan's proof \cite{Car} of the strong completeness of the cone $\mathcal E_{\kappa_2}^+(\mathbb R^n)$ of all positive measures on $\mathbb R^n$ with finite Newtonian energy to a perfect kernel $\kappa$ on a locally compact space $X$ and suitable classes of ({\it signed\/}) measures $\mu\in\mathcal E_\kappa(X)$. Theorem~\ref{pr:nons} below, pertaining to the Newtonian kernel, shows that assumption (\ref{eq-intr}) is essential not only for the proofs in \cite{ZPot2,ZPot3}, but also for the validity of the approach developed therein. Omitting now (\ref{eq-intr}), in the present paper we have nevertheless succeeded in working out a substantive theory for the Riesz kernels by imposing instead an appropriate upper constraint on the vector measures under consideration.\end{remark}

While our investigation is focused on theoretical aspects in a very general context, and possible applications are so far outside the frames of the present paper, it is noteworthy to remark that minimum energy problems in the constrained and unconstrained settings for the logarithmic kernel on $\mathbb R^n$, also referred to as 'vector equilibrium problems', have been considered for several decades in relation to Hermite--Pad\'{e} approximants \cite{GR,A} and random matrix ensembles \cite{Ku,AK}. See also \cite{HK,BKMW,Y,LT} and the references therein.

\section{Preliminaries}\label{sec-pr}

Let $X$ be a locally compact (Hausdorff) space, to be specified below, and $\mathfrak M(X)$ the linear
space of all real-valued scalar Radon measures $\mu$ on $X$, equipped with the {\it vague\/} topology, i.e.\ the topology of
pointwise convergence on the class $C_0(X)$ of all real-valued continuous functions on $X$ with compact
support.\footnote{When speaking of a continuous numerical function we understand that the values are {\it finite\/} real numbers.} The vague topology on $\mathfrak M(X)$ is Hausdorff; hence, a vague limit of any sequence (net) in $\mathfrak M(X)$ is unique (whenever it exists).
These and other notions and results from the theory of measures and integration on a locally compact space, to be used throughout the paper, can be found in \cite{E2,B2} (see also \cite{F1} for a short survey).
We denote by $\mu^+$ and $\mu^-$ the positive and the negative parts in the Hahn--Jordan decomposition of a measure $\mu\in\mathfrak M(X)$ and by $S^\mu_{X}=S(\mu)$ its support. A measure $\mu$ is said to be {\it bounded\/} if $|\mu|(X)<\infty$ where $|\mu|:=\mu^++\mu^-$. Let $\mathfrak M^+(X)$ stand for the (convex, vaguely closed) cone of all positive $\mu\in\mathfrak M(X)$, and let $\Psi(X)$ consist of all lower semicontinuous (l.s.c.)
functions $\psi: X\to(-\infty,\infty]$, nonnegative unless $X$ is compact.
The following well known fact (see e.g.\ \cite[Section~1.1]{F1}) will often be used.

\begin{lemma}\label{lemma-semi}For any\/ $\psi\in\Psi(X)$, $\mu\mapsto\langle\psi,\mu\rangle:=\int\psi\,d\mu$ is vaguely l.s.c.\ on\/ $\mathfrak M^+(X)$.\footnote{Throughout the paper the integrals are understood as {\it upper\/} integrals~\cite{B2}.}\end{lemma}

A {\it kernel\/} $\kappa(x,y)$ on $X$ is defined as a symmetric function from $\Psi(X\times X)$. Given $\mu,\mu_1\in\mathfrak M(X)$, we denote by
$\kappa(\mu,\mu_1)$ and $\kappa(\cdot,\mu)$ the {\it mutual
energy\/} and the {\it potential\/} relative to the kernel $\kappa$,
respectively, i.e.\footnote{When introducing notation of a numerical value, we assume
the corresponding object on the right to be well defined (as a finite number or $\pm\infty$).}
\begin{align*}
\kappa(\mu,\mu_1)&:=\iint\kappa(x,y)\,d\mu(x)\,d\mu_1(y),\\
\kappa(x,\mu)&:=\int\kappa(x,y)\,d\mu(y),\quad x\in X.
\end{align*}
Observe that $\kappa(x,\mu)$ is well defined provided that $\kappa(x,\mu^+)$ or $\kappa(x,\mu^-)$ is finite, and then $\kappa(x,\mu)=\kappa(x,\mu^+)-\kappa(x,\mu^-)$. In particular, if $\mu\in\mathfrak M^+(X)$ then $\kappa(\cdot,\mu)$ is defined everywhere and represents a l.s.c.\ function on $X$, bounded from below (see Lemma~\ref{lemma-semi}). Also note that $\kappa(\mu,\mu_1)$ is well defined and equal to $\kappa(\mu_1,\mu)$ provided that
$\kappa(\mu^+,\mu_1^+)+\kappa(\mu^-,\mu_1^-)$ or $\kappa(\mu^+,\mu_1^-)+\kappa(\mu^-,\mu_1^+)$ is finite.

For $\mu=\mu_1$ the mutual energy $\kappa(\mu,\mu_1)$ becomes the
{\it energy\/} $\kappa(\mu,\mu)$. Let $\mathcal E_\kappa(X)$ consist
of all $\mu\in\mathfrak M(X)$ whose energy $\kappa(\mu,\mu)$ is finite, which by definition means that $\kappa(\mu^+,\mu^+)$, $\kappa(\mu^-,\mu^-)$ and $\kappa(\mu^+,\mu^-)$ are all finite, and let $\mathcal E^+_\kappa(X):=\mathcal E_\kappa(X)\cap\mathfrak M^+(X)$.

For a set $Q\subset X$ let $\mathfrak M^+(Q)$ consist of all $\mu\in\mathfrak M^+(X)$ {\it carried by\/} $Q$, which means that $Q^c:=X\setminus Q$ is locally $\mu$-negligible, or equivalently that $Q$ is $\mu$-meas\-ur\-able and $\mu=\mu|_Q$ where $\mu|_Q=1_Q\cdot\mu$ is the trace (restriction) of $\mu$ on $Q$ \cite[Chapter~V, Section~5, n$^\circ$\,3, Example]{B2}. (Here $1_Q$ denotes the indicator function of $Q$.) If $Q$ is closed, then $\mu$ is carried by $Q$ if and only if it is supported by $Q$, i.e.\ $S(\mu)\subset Q$. Also note that if either $X$ is {\it countable at infinity\/} (i.e.\ $X$ can be represented as a countable union of compact sets \cite[Chapter~I, Section~9, n$^\circ$\,9]{B1}), or $\mu$ is bounded, then the concept of local $\mu$-negligibility coincides with that of $\mu$-negligibility; and hence $\mu\in\mathfrak M^+(Q)$ if and only if $\mu^*(Q^c)=0$, $\mu^*(\cdot)$ being the {\it outer measure\/} of a set.
Write $\mathcal E_\kappa^+(Q):=\mathcal E_\kappa(X)\cap\mathfrak M^+(Q)$, $\mathfrak M^+(Q,q):=\{\mu\in\mathfrak M^+(Q): \mu(Q)=q\}$ and $\mathcal E_\kappa^+(Q,q):=\mathcal E_\kappa(X)\cap\mathfrak M^+(Q,q)$, where $q\in(0,\infty)$.

In the rest of this section and throughout Sections~\ref{sec-cond}--\ref{sec-Gauss-c} a kernel $\kappa$ is assumed to be {\it strictly positive definite\/}, which means that the energy $\kappa(\mu,\mu)$, $\mu\in\mathfrak M(X)$, is nonnegative whenever defined, and it equals $0$ only for $\mu=0$.
Then $\mathcal E_\kappa(X)$ forms a pre-Hil\-bert space with the inner product $\kappa(\mu,\mu_1)$ and the energy norm
$\|\mu\|_{\mathcal E_\kappa(X)}:=\|\mu\|_\kappa:=\sqrt{\kappa(\mu,\mu)}$ \cite{F1}. The (Hausdorff) topology
on $\mathcal E_\kappa(X)$ defined by $\|\cdot\|_\kappa$ is termed {\it strong\/}.

In contrast to \cite{Fu4,Fu5} where a capacity has been treated as a functional acting on positive numerical functions on $X$, in the present study we use the (standard) concept of capacity as a set function. Thus the ({\it inner\/}) {\it capacity\/} of a set $Q\subset X$ relative to the kernel $\kappa$, denoted $c_\kappa(Q)$, is defined by
\begin{equation}\label{cap-def}c_\kappa(Q):=\bigl[\inf_{\mu\in\mathcal
E_\kappa^+(Q,1)}\,\kappa(\mu,\mu)\bigr]^{-1}\end{equation}
(see e.g.\ \cite{F1,O}). Then $0\leqslant c_\kappa(Q)\leqslant\infty$. (As usual, the
infimum over the empty set is taken to be $+\infty$. We also set $1\bigl/(+\infty)=0$ and $1\bigl/0=+\infty$.)

An assertion $\mathcal U(x)$ involving a variable point $x\in X$ is said to hold $c_\kappa$-{\it n.e.}\ on $Q$ if $c_\kappa(N)=0$ where $N$ consists of all $x\in Q$ for which $\mathcal U(x)$ fails to hold. We shall use the short form 'n.e.' instead of '$c_\kappa$-n.e.' if this will not cause any mis\-under\-standing.

\begin{definition}\label{def-perf}Following~\cite{F1}, we call a (strictly positive definite)
kernel $\kappa$ {\it perfect\/} if every strong Cauchy sequence in $\mathcal E_\kappa^+(X)$ converges strongly to any of its vague cluster points.\footnote{It follows from Theorem~\ref{fu-complete} that for a perfect kernel such a vague cluster point exists and is unique.}\end{definition}

\begin{remark}On $X=\mathbb R^n$, $n\geqslant3$, the Riesz kernel $\kappa_\alpha(x,y)=|x-y|^{\alpha-n}$, $\alpha\in(0,n)$, is strictly positive definite and moreover perfect \cite{D1,D2}, and hence so is the Newtonian kernel $\kappa_2(x,y)=|x-y|^{2-n}$ \cite{Car}. Recently it has been shown that if $X=D$ where $D$ is an arbitrary open set in
$\mathbb R^n$, $n\geqslant3$, and $G^\alpha_D$, $\alpha\in(0,2]$, is the $\alpha$-Green kernel on $D$ \cite[Chapter~IV, Section~5]{L}, then $\kappa=G^\alpha_D$ is strictly positive definite and moreover perfect \cite[Theorems~4.9, 4.11]{FZ}. The restriction of the logarithmic
kernel $-\log\,|x-y|$ on $\mathbb R^2$ to the closed
disk $\overline{B}(0,r):=\bigl\{\xi\in\mathbb R^2: \ |\xi|\leqslant r<1\bigr\}$ is perfect as well.\footnote{Indeed, the restriction of the logarithmic
kernel to $\overline{B}(0,r)$, $r<1$, is strictly positive definite by \cite[Theorem~1.16]{L}. Since it satisfies Frostman's maximum principle \cite[Theorem~1.6]{L}, it is regular according to \cite[Eq.~1.3]{O}, and hence perfect by~\cite{O1} (see also \cite[Theorem~3.4.1]{F1}).}\end{remark}

\begin{theorem}[{\rm see \cite{F1}}]\label{fu-complete} If the kernel\/ $\kappa$ is perfect, then the cone\/ $\mathcal E_\kappa^+(X)$ is strongly complete and the strong topology on\/ $\mathcal E_\kappa^+(X)$ is finer than the\/ {\rm(}induced\/{\rm)} vague topology on\/ $\mathcal E_\kappa^+(X)$.\end{theorem}

\begin{remark}\label{rem-net} When speaking of the vague topology, one has to consider {\it nets\/} or {\it filters\/}
in $\mathfrak M(X)$ instead of sequences since the vague topology in
general does not satisfy the first axiom of countability. We follow
Moore and Smith's theory of convergence \cite{MS}, based on the concept of
nets (see also \cite[Chapter~2]{K} and \cite[Chapter~0]{E2}). However, if $X$ is metrizable and countable at infinity, then $\mathfrak M^+(X)$ satisfies the first axiom of countability, and the use of nets may be avoided. Indeed, if $\varrho(\cdot,\cdot)$ denotes a metric on $X$, then a countable base $(\mathcal V_k)_{k\in\mathbb N}$ of vague neighborhoods of a measure $\mu_0\in\mathfrak M^+(X)$ can be obtained for example as follows after choosing a countable dense sequence $\{x_k\}_{k\in\mathbb N}$ of points of $X$:
\[\mathcal V_k=\Bigl\{\mu\in\mathfrak M^+(X): \ \int\bigl(1-k\varrho(x_k,x)\bigr)^+\,d|\mu-\mu_0|(x)<1/k\Bigr\}.\]
(The existence of such $\{x_k\}_{k\in\mathbb N}$ for $X$ in question is ensured by \cite[Chapter~IX, Section~2, n$^\circ$\,8, Proposition~12]{B1'} and \cite[Chapter~IX, Section~2, n$^\circ$\,9, Corollary to Proposition~16]{B1'}.)
\end{remark}

\begin{remark}\label{remma}In contrast to Theorem~\ref{fu-complete}, for a perfect kernel $\kappa$ the whole pre-Hil\-bert space $\mathcal E_\kappa(X)$ is in general strongly {\it incomplete\/}, and this is the case even for the $\alpha$-Riesz kernel of order $\alpha\in(1,n)$ on $\mathbb R^n$, $n\geqslant 3$
(see~\cite{Car} and \cite[Theorem~1.19]{L}). Compare with \cite[Theorem~1]{ZUmzh} where the strong completeness has been established for the metric subspace of all ({\it signed\/}) $\nu\in\mathcal E_{\kappa_\alpha}(\mathbb R^n)$, $\alpha\in(0,n)$, such that $\nu^+$ and $\nu^-$ are supported by closed nonintersecting sets  $F_1,F_2\subset\mathbb R^n$, respectively. This result from \cite{ZUmzh} was proved with the aid of Deny's theorem~\cite{D1} stating that $\mathcal E_{\kappa_\alpha}(\mathbb R^n)$ can be completed by making use of tempered distributions on $\mathbb R^n$ with finite energy, defined in terms of Fourier transforms.\end{remark}

\begin{remark}\label{remark}The concept of perfect kernel is an efficient tool in minimum energy problems
over classes of {\it positive scalar\/} Radon measures on $X$ with
finite energy. Indeed, if $Q\subset X$ is closed, $c_\kappa(Q)\in(0,+\infty)$, and $\kappa$ is perfect, then the problem (\ref{cap-def}) has a unique solution $\lambda$ \cite[Theorem~4.1]{F1}; we shall call such $\lambda$ the ({\it inner\/}) {\it $\kappa$-capacitary measure\/} on $Q$.
Later the concept of perfectness has been shown to be efficient in minimum energy problems over classes of {\it vector measures\/} associated with a standard condenser $\mathbf A$ in $X$ (see \mbox{\cite{ZPot1}--\cite{ZPot3}}, where $\kappa$ and $\mathbf A$ were assumed to satisfy (\ref{eq-intr})). See Remark~\ref{rem-intr} above for some details of the approach developed in \mbox{\cite{ZPot1}--\cite{ZPot3}}; compare with the above-mentioned \cite[Theorem~1]{ZUmzh} where condition (\ref{eq-intr}) has not been required. See also  Remarks~\ref{r-3} and~\ref{rem-3} below for a short survey of the results obtained in \mbox{\cite{ZPot1}--\cite{ZPot3}}.\end{remark}

\section{Vector measures associated with a generalized condenser}\label{sec-cond}

\subsection{Generalized condensers} Fix a finite set $I$ of indices $i\in\mathbb N$ and an ordered collection $\mathbf A:=(A_i)_{i\in I}$ of nonempty closed sets $A_i\subset X$, $X$ being a locally compact space, where each $A_i$, $i\in I$, has the sign $s_i:={\rm sign}\,A_i=\pm1$ prescribed. Let $I^+$ consist of all $i\in I$ such that $s_i=+1$, and $I^-:=I\setminus I^+$. The sets $A_i$, $i\in I^+$, and $A_j$, $j\in I^-$, are termed the {\it positive\/} and the {\it negative\/} plates of the collection $\mathbf A$.
Write
\[A^+:=\bigcup_{i\in I^+}\,A_i,\quad A^-:=\bigcup_{j\in I^-}\,A_j,\quad A:=A^+\cup A^-,\quad\delta_{\mathbf A}:=A^+\cap A^-.\]

\begin{definition}\label{def-cond} $\mathbf A=(A_i)_{i\in I}$ is said to be a {\it standard condenser\/} in $X$ if $\delta_{\mathbf A}=\varnothing$.
\end{definition}

Note that any two equally signed plates of a standard condenser may intersect each other over a set of nonzero capacity (or even coincide).

Fix a (strictly positive definite) kernel $\kappa$ on $X$. By relaxing in the above definition the requirement $\delta_{\mathbf A}=\varnothing$, we slightly generalize the notion of standard condenser as follows.

\begin{definition}\label{def-cond-st} $\mathbf A=(A_i)_{i\in I}$ is said to be a {\it generalized condenser\/} in $X$ if $c_\kappa(\delta_{\mathbf A})=0$.\footnote{Gonchar and Rakhmanov \cite{GR} seem to be the first to consider such a generalization.}
\end{definition}

A (generalized) condenser $\mathbf A$ is said to be {\it compact\/} if all the $A_i$, $i\in I$, are compact, and {\it noncompact\/} if at least one of the $A_i$, $i\in I$, is noncompact.

In Examples~\ref{ex-1} and \ref{ex-1'} below, $X=\mathbb R^n$ with $n\geqslant3$. Let $B(x,r)$, respectively $\overline{B}(x,r)$, denote the open, respectively closed, $n$-dimensional ball of radius $r$ centered at $x\in\mathbb R^n$. We shall also write $S(x,r):=\partial_{\mathbb R^n}B(x,r)$.

\begin{example}\label{ex-1}Let $I^+:=\{1\}$ and $I^-:=\{2,3,4\}$. Define $A_1:=\overline{B}(\xi_1,1)$, $A_2:=\overline{B}(\xi_2,1)$, $A_3:=\overline{B}(\xi_3,2)$ and $A_4:=\overline{B}(\xi_4,1)$ where $\xi_1=(0,0,\ldots,0)$, $\xi_2=(2,0,\ldots,0)$, $\xi_3=(3,0,\ldots,0)$ and $\xi_4=(-2,0,\ldots,0)$. Since $\delta_{\mathbf A}$ consists of the points $\xi_5=(-1,0,\ldots,0)$ and $\xi_6=(1,0,\ldots,0)$, $\mathbf A=(A_i)_{i\in I}$ forms a generalized condenser in $\mathbb R^n$ for any (strictly positive definite) kernel $\kappa$ on $\mathbb R^n$ with the property that $\kappa(x,y)=\infty$ whenever $x=y$.\end{example}

See Example \ref{ex-2} for kernels and constraints under which the constrained minimum energy problem (Problem~\ref{pr2}) for such a condenser
admits a solution (has no short-circuit) despite the two touching points for the oppositely charged plates.

\begin{figure}[htbp]
\begin{center}
\includegraphics[width=4.5in]{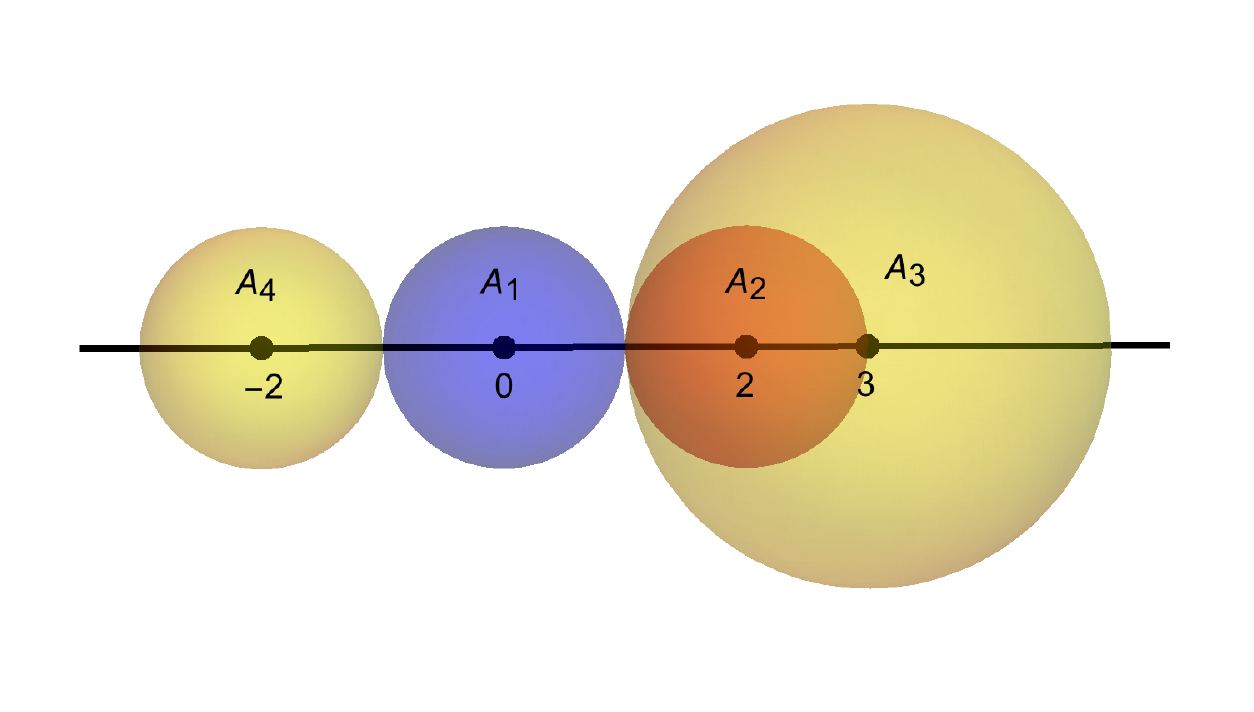}
\vspace{-.3in}
\caption{Generalized condenser of Example~\ref{ex-1}}
\label{fig3.3}
\end{center}
\end{figure}

\begin{example}\label{ex-1'} Assume that $n=3$, $I^+:=\{1\}$, $I^-:=\{2\}$, and let
\begin{align*}
A_1&:=\bigl\{x\in\mathbb R^3: \ 1\leqslant x_1<\infty, \ x_2^2+x_3^2=\exp(-2x_1^{r_1})\bigr\},\\
A_2&:=\bigl\{x\in\mathbb R^3: \ 2\leqslant x_1<\infty, \ x_2^2+x_3^2=\exp(-2x_1^{r_2})\bigr\}\end{align*}
where $1<r_1<r_2<\infty$. Then $A_1$ and $A_2$ form a standard condenser in $\mathbb R^3$ such that
\[{\rm dist}\,(A_1,A_2):=\inf_{x\in A_1, \ y\in A_2}\,|x-y|=0.\]
\end{example}

See Example \ref{ex-2'} for a kernel and constraints under which the constrained minimum energy problem (Problem~\ref{pr2}) for such a condenser admits a solution (has no short-circuit) despite the touching point at infinity.

\begin{figure}[htbp]
\begin{center}
\includegraphics[width=4.5in]{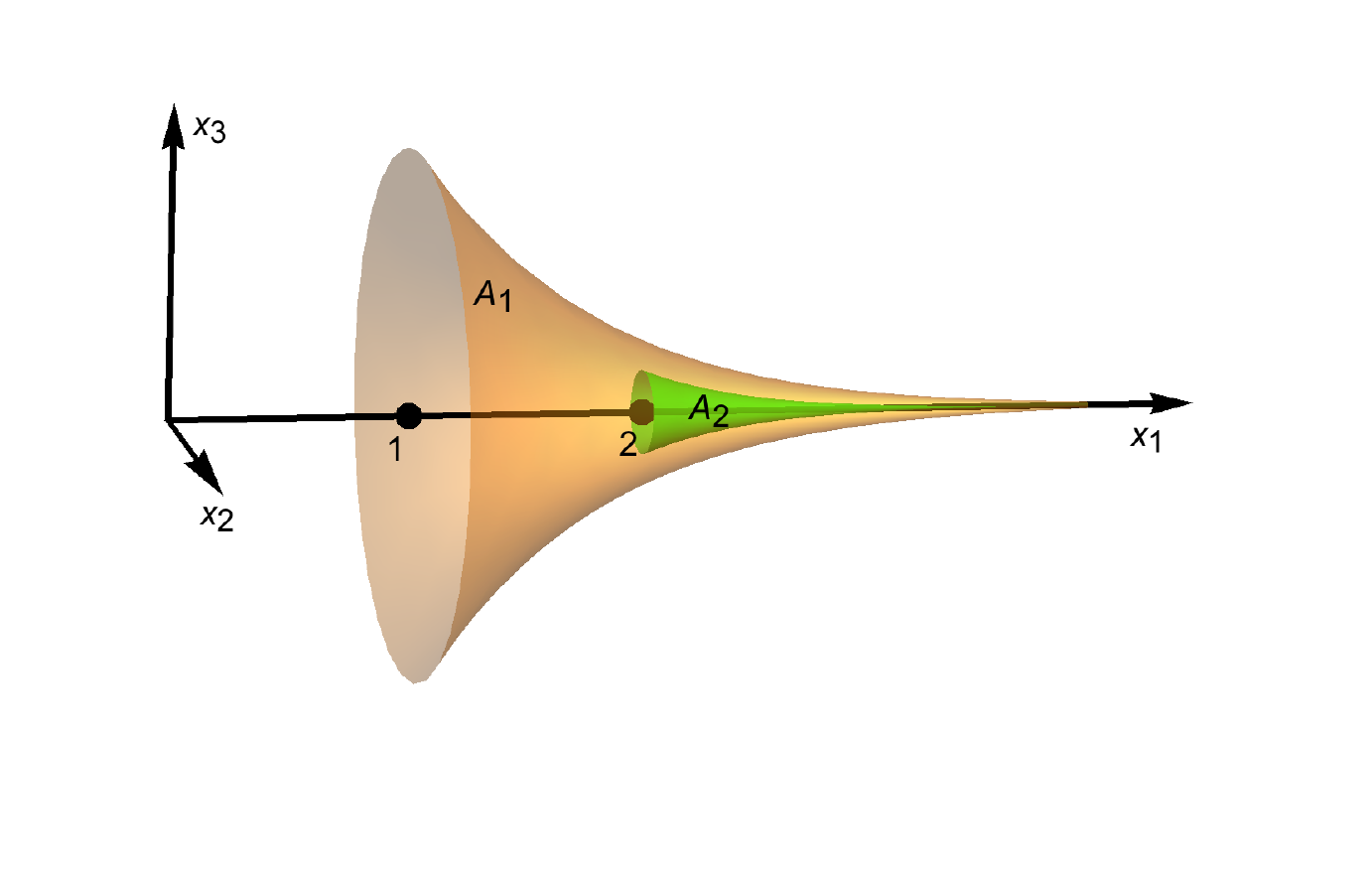}
\vspace*{-.5in}
\caption{Generalized condenser of Example~\ref{ex-1'}}
\label{fig3.3'}
\end{center}
\end{figure}

\subsection{Vector measures associated with a condenser $\mathbf A$. Vague topology} In the rest of the paper, fix a generalized condenser $\mathbf A=(A_i)_{i\in I}$ in $X$, and let $\mathfrak M^+(\mathbf A)$ stand for the Cartesian product $\prod_{i\in I}\,\mathfrak M^+(A_i)$. Then $\boldsymbol{\mu}\in\mathfrak M^+(\mathbf A)$ is a positive {\it vector measure\/} $(\mu^i)_{i\in I}$ with the components $\mu^i\in\mathfrak M^+(A_i)$; such $\boldsymbol{\mu}$ is said to be {\it associated\/} with~$\mathbf A$.

\begin{definition}\label{def-vague}
The {\it vague topology\/} on $\mathfrak M^+(\mathbf A)$ is the topology of the product space $\prod_{i\in I}\,\mathfrak M^+(A_i)$ where each of the factors $\mathfrak M^+(A_i)$, $i\in I$, is endowed with the vague topology induced from $\mathfrak M(X)$. Namely, a net $(\boldsymbol\mu_s)_{s\in S}\subset\mathfrak M^+(\mathbf A)$ converges to $\boldsymbol\mu$ {\it vaguely\/} if for every $i\in I$, $\mu_s^i\to\mu^i$ vaguely in $\mathfrak M(X)$ when $s$ increases along~$S$.\end{definition}

As all the $A_i$, $i\in I$, are closed in $X$, $\mathfrak M^+(\mathbf A)$ is vaguely closed in $\mathfrak M^+(X)^{|I|}$ where $|I|:={\rm Card}\,I$. Furthermore, since every $\mathfrak M^+(A_i)$ is Hausdorff in the vague topology, so is $\mathfrak M^+(\mathbf A)$ \cite[Chapter~3, Theorem~5]{K}. Hence, a va\-gue limit of any net in $\mathfrak M^+(\mathbf A)$ belongs to $\mathfrak M^+(\mathbf A)$ and is unique (provided that the vague limit exists).

Given $\boldsymbol{\mu}\in\mathfrak M^+(\mathbf A)$ and a vector-valued function
$\boldsymbol{u}=(u_i)_{i\in I}$ with $u_i:X\to[-\infty,\infty]$ such that each $\int u_i\,d\mu^i$ as well as their sum over $i$ exist (as finite numbers or $\pm\infty$), write
\begin{equation}\label{ul}\langle\boldsymbol{u},\boldsymbol{\mu}\rangle:=\sum_{i\in I}\,\langle u_i,\mu^i\rangle=\sum_{i\in I}\,\int u_i\,d\mu^i.\end{equation}

Let $\mathcal E^+_\kappa(\mathbf A)$ consist of all $\boldsymbol{\mu}\in\mathfrak M^+(\mathbf A)$ such that $\kappa(\mu^i,\mu^i)<\infty$ for all $i\in I$; in other words, $\mathcal E^+_\kappa(\mathbf A):=\prod_{i\in I}\,\mathcal E^+_\kappa(A_i)$.
For any $\boldsymbol{\mu}\in\mathcal E^+_\kappa(\mathbf A)$, $\mu^i(\delta_{\mathbf A})=0$ by \cite[Lemma~2.3.1]{F1}. Hence each of the $i$-components $\mu^i$, $i\in I$, of $\boldsymbol{\mu}\in\mathcal E^+_\kappa(\mathbf A)$ is carried by $A_i^\delta$, where
\begin{equation}\label{delta}A_i^\delta:=A_i\setminus\delta_{\mathbf A},\end{equation} though the support of $\mu^i$ may coincide with the whole $A_i$. We thus actually have
\[\mathcal E^+_\kappa(\mathbf A)=\prod_{i\in I}\,\mathcal E^+_\kappa(A_i^\delta).\]

Write
$A^+_\delta:=A^+\setminus\delta_{\mathbf A}$ and $A^-_\delta:=A^-\setminus\delta_{\mathbf A}$.
For any $\boldsymbol{\mu}\in\mathcal E^+_\kappa(\mathbf A)$ define
\begin{equation*}R\boldsymbol{\mu}:=R_{\mathbf A}\boldsymbol{\mu}:=\sum_{i\in I}\,s_i\mu^i,\end{equation*}
the 'resultant' of $\boldsymbol{\mu}$.
Since $A^+_\delta\cap A^-_\delta=\varnothing$, $R\boldsymbol{\mu}$ is a ({\it signed\/}) scalar Radon measure on $X$ whose positive and negative parts, carried respectively by $A^+_\delta$ and $A^-_\delta$, are given by
\begin{equation}\label{defR}(R\boldsymbol{\mu})^+:=\sum_{i\in I^+}\,\mu^i\text{ \ and \ }(R\boldsymbol{\mu})^-:=\sum_{i\in I^-}\,\mu^i.\end{equation}
If $\boldsymbol{\mu}=\boldsymbol{\mu}_1$ where $\boldsymbol{\mu},\boldsymbol{\mu}_1\in\mathcal E^+_\kappa(\mathbf A)$, then $R\boldsymbol{\mu}=R\boldsymbol{\mu}_1$, but not the other way around.

\begin{lemma}\label{inj} For the mapping\/ $\boldsymbol\mu\mapsto R\boldsymbol\mu$ to be injective it is necessary and sufficient that all the\/ $A_i$, $i\in I$, be mutually essentially disjoint, i.e.\ with\/
$c_\kappa(A_i\cap A_j)=0$ for all\/ $i\ne j$.
\end{lemma}

\begin{proof} Since a nonzero positive scalar measure of finite energy does not charge any set of zero capacity \cite[Lemma~2.3.1]{F1}, the sufficiency part of the lemma is obvious. To prove the necessity part, assume on the contrary that there are two equally signed plates $A_k$ and $A_\ell$, $k\ne\ell$, with $c_\kappa(A_k\cap A_\ell)>0$. By \cite[Lemma~2.3.1]{F1}, there exists a nonzero positive scalar measure $\tau\in\mathcal E_\kappa^+(A_k\cap A_\ell)$. Choose $\boldsymbol\mu=(\mu^i)_{i\in I}\in\mathcal E^+_\kappa(\mathbf A)$ such that $\mu^k|_{A_k\cap A_\ell}-\tau\geqslant0$, and define $\boldsymbol\mu_m=(\mu_m^i)_{i\in I}\in\mathcal E^+_\kappa(\mathbf A)$, $m=1,2$, where $\mu_1^k=\mu^k-\tau$ and $\mu_1^i=\mu^i$ for all $i\ne k$, while $\mu_2^\ell=\mu^\ell+\tau$ and $\mu_2^i=\mu^i$ for all $i\ne\ell$. Then $R\boldsymbol\mu_1=R\boldsymbol\mu_2$, but $\boldsymbol\mu_1\ne\boldsymbol\mu_2$.\end{proof}

We call $\boldsymbol{\mu},\boldsymbol{\mu}_1\in\mathcal E^+_\kappa(\mathbf A)$ {\it$R$-equivalent\/} if $R\boldsymbol{\mu}=R\boldsymbol{\mu}_1$. For $\boldsymbol{\mu}\in\mathcal E^+_\kappa(\mathbf A)$, let $[\boldsymbol{\mu}]$ consist of all $\boldsymbol{\mu}_1\in\mathcal E^+_\kappa(\mathbf A)$ that are $R$-equivalent to $\boldsymbol{\mu}$. Note that $\boldsymbol{\mu}=\boldsymbol{0}$ is the only element of $[\boldsymbol{0}]$.

\subsection{A semimetric structure on classes of vector measures}\label{sec-Metric}
To avoid trivialities, for a given (generalized) condenser $\mathbf A=(A_i)_{i\in I}$ and a given (strictly positive definite) kernel $\kappa$ on a locally compact space $X$ we shall always require that
\begin{equation}\label{cnon0}c_\kappa(A_i)>0\quad\text{for all \ }i\in I.\end{equation}
In accordance with an electrostatic interpretation of a condenser, we say that the interaction between the components $\mu^i$, $i\in I$, of $\boldsymbol\mu\in\mathcal E^+_\kappa(\mathbf A)$ is characterized by the matrix $(s_is_j)_{i,j\in I}$, where $s_i:={\rm sign}\,A_i$.
Given $\boldsymbol{\mu},\boldsymbol{\mu}_1\in\mathcal E^+_\kappa(\mathbf A)$, we define the {\it mutual energy\/}
\begin{equation}\label{env}\kappa(\boldsymbol{\mu},\boldsymbol{\mu}_1):=\sum_{i,j\in I}\,s_is_j\kappa(\mu^i,\mu_1^j)\end{equation}
and the {\it vector potential\/} $\kappa_{\boldsymbol{\mu}}(\cdot)$ as a vector-valued function on $X$ with the components
\begin{equation}\label{potv}\kappa_{\boldsymbol{\mu}}^i(\cdot):=\sum_{j\in I}\,s_is_j\kappa(\cdot,\mu^j),\quad i\in I.\end{equation}

\begin{lemma}\label{l-Rpot} For any\/ $\boldsymbol{\mu}\in\mathcal E_\kappa^+(\mathbf A)$, the\/ $\kappa_{\boldsymbol{\mu}}^i(\cdot)$, $i\in I$, are well defined and finite n.e.\ on\/~$X$.\end{lemma}

\begin{proof}Since $\mu^i\in\mathcal E_\kappa^+(X)$ for every $i\in I$, $\kappa(\cdot,\mu^i)$ is finite n.e.\ on $X$ \cite[p.~164]{F1}. Furthermore, the set of all $x\in X$ with $\kappa(x,\mu^i)=\infty$ is universally measurable, for $\kappa(\cdot,\mu^i)$ is l.s.c.\ on $X$. Combined with the fact that the inner capacity $c_\kappa(\cdot)$ is subadditive on universally measurable sets \cite[Lemma~2.3.5]{F1}, this proves the lemma.\end{proof}

\begin{lemma}\label{l-Ren} For any\/
$\boldsymbol{\mu},\boldsymbol{\mu}_1\in\mathcal E^+_\kappa(\mathbf A)$ we have
\begin{equation}\label{Re}\kappa(\boldsymbol{\mu},\boldsymbol{\mu}_1)=\kappa(R\boldsymbol{\mu},R\boldsymbol{\mu}_1)\in(-\infty,\infty).\end{equation}\end{lemma}

\begin{proof} This is obtained directly from relations (\ref{defR}) and (\ref{env}).\end{proof}

For $\boldsymbol{\mu}=\boldsymbol{\mu}_1\in\mathcal E_\kappa^+(\mathbf A)$ the mutual energy $\kappa(\boldsymbol{\mu},\boldsymbol{\mu}_1)$ becomes the {\it energy\/} $\kappa(\boldsymbol{\mu},\boldsymbol{\mu})$ of $\boldsymbol{\mu}$. By the strict positive definiteness of the kernel $\kappa$, we see from Lemma~\ref{l-Ren} that  $\kappa(\boldsymbol{\mu},\boldsymbol{\mu})$, $\boldsymbol{\mu}\in\mathcal E_\kappa^+(\mathbf A)$, is always ${}\geqslant0$, and it is zero only for $\boldsymbol{\mu}=\boldsymbol{0}$.

In order to introduce a (semi)metric structure on the cone $\mathcal E_\kappa^+(\mathbf A)$, we define
\begin{equation}\label{isom}\|\boldsymbol{\mu}-\boldsymbol{\mu}_1\|_{\mathcal E^+_\kappa(\mathbf A)}:=\|R\boldsymbol{\mu}-R\boldsymbol{\mu}_1\|_\kappa\text{ \ for all  \ }\boldsymbol{\mu},\boldsymbol{\mu}_1\in\mathcal E^+_\kappa(\mathbf A).\end{equation}
Based on (\ref{Re}), we see by straightforward calculation that, in fact,
\begin{equation}\label{metr-def}\|\boldsymbol{\mu}-\boldsymbol{\mu}_1\|^2_{\mathcal E^+_\kappa(\mathbf A)}=\sum_{i,j\in I}\,s_is_j\kappa(\mu^i-\mu_1^i,\mu^j-\mu_1^j).\end{equation}
On account of Lemma~\ref{inj}, we are thus led to the following conclusion.

\begin{theorem}\label{lemma:semimetric} $\mathcal E^+_\kappa(\mathbf A)$ is a semimetric space with
the semimetric defined by either of the\/ {\rm(}equivalent\/{\rm)} relations\/ {\rm(\ref{isom})} or\/ {\rm(\ref{metr-def})}, and this
space is isometric to its\/ $R$-image in\/ $\mathcal E_\kappa(X)$. The semimetric\/ $\|\boldsymbol\mu-\boldsymbol\mu_1\|_{\mathcal E^+_\kappa(\mathbf A)}$ is a metric on\/ $\mathcal E^+_\kappa(\mathbf A)$ if and only if all the\/ $A_i$, $i\in I$, are mutually essentially disjoint.
\end{theorem}

Similar to the terminology in the pre-Hilbert space $\mathcal E_\kappa(X)$, we therefore call the topology of the semimetric space $\mathcal E^+_\kappa(\mathbf A)$ {\it strong\/}.
Now \cite[Theorem~1 and Corollary~1]{ZUmzh}, mentioned in Remark~\ref{remma} above, can be rewritten as follows.

\begin{theorem}\label{ZU}If\/ $\mathbf A=(A_1,A_2)$ is a standard condenser in\/ $\mathbb R^n$, $n\geqslant3$, with\/ $s_1s_2=-1$ and\/ $\kappa_\alpha$ is the Riesz kernel of an arbitrary order\/ $\alpha\in(0,n)$, then the metric space\/ $\mathcal E_{\kappa_\alpha}^+(\mathbf A)$ is strongly complete, and the strong topology on this space is finer than the vague topology.\end{theorem}

Note that, under the assumptions of Theorem~\ref{ZU}, $\kappa_\alpha$ may be unbounded on $A_1\times A_2$ (compare with Remarks~\ref{rem-intr}, \ref{remark}, \ref{r-3}, and~\ref{rem-3}).

\section{Unconstrained $\mathbf{f}$-weighted minimum energy problem}\label{sec-Gauss}

For a (strictly positive definite) kernel $\kappa$ on $X$ and a (generalized) condenser $\mathbf A=(A_i)_{i\in I}$, we shall consider minimum energy problems with an external field over certain subclasses of $\mathcal E^+_\kappa(\mathbf A)$, to be defined below. Since the admissible  measures in those problems are of finite energy, there is no loss of generality in assuming that each $A_i$ coincides with its {\it $\kappa$-reduced kernel\/} \cite[p.~164]{L}, which consists of all $x\in A_i$ such that $c_\kappa(A_i\cap U_x)>0$ for every neighborhood $U_x$ of $x$ in~$X$.

Fix a vector-valued function $\mathbf{f}=(f_i)_{i\in I}$, where each $f_i:X\to[-\infty,\infty]$ is $\mu$-measurable for every $\mu\in\mathfrak M^+(X)$ and treated as an {\it external field\/} acting on the charges (measures) from $\mathcal E^+_\kappa(A_i)$. The {\it $\mathbf f$-weighted vector potential\/} and
the {\it $\mathbf f$-weighted energy\/} of $\boldsymbol{\mu}\in\mathcal E_\kappa^+(\mathbf A)$ are defined respectively by\footnote{$G_{\kappa,\mathbf{f}}(\mathbf{\cdot})$ is also known as the {\it Gauss functional\/} (see e.g.\ \cite{O}). Note that when defining $G_{\kappa,\mathbf{f}}(\mathbf{\cdot})$, we have used the notation (\ref{ul}).}
\begin{align}\label{wpot}\mathbf W_{\kappa,\mathbf{f}}^{\boldsymbol{\mu}}&:=\kappa_{\boldsymbol{\mu}}+\mathbf f,\\
\label{wen}G_{\kappa,\mathbf{f}}(\boldsymbol{\mu})&:=\kappa(\boldsymbol{\mu},\boldsymbol{\mu})+2\langle\mathbf{f},\boldsymbol{\mu}\rangle.\end{align}
Let $\mathcal E_{\kappa,\mathbf{f}}^+(\mathbf A)$ consist of all $\boldsymbol{\mu}\in\mathcal E_\kappa^+(\mathbf A)$ with finite $G_{\kappa,\mathbf{f}}(\boldsymbol{\mu})$ (equivalently, with finite $\langle\mathbf{f},\boldsymbol{\mu}\rangle$).

In this paper we shall tacitly assume that either Case~I or Case~II holds, where:\footnote{The notation $\Psi(X)$ has been introduced at the end of the first paragraph in Section~\ref{sec-pr}.}
\begin{itemize}
\item[\rm I.] {\it For every\/ $i\in I$, $f_i\in\Psi(X)$};
\item[\rm II.] {\it For every\/ $i\in I$, $f_i(x)=s_i\kappa(x,\zeta)$ where a\/ {\rm(}signed\/{\rm)} measure\/ $\zeta\in\mathcal E_\kappa(X)$ is given\/}.
\end{itemize}
For any $\boldsymbol{\mu}\in\mathcal E^+_{\kappa}(\mathbf A)$, $G_{\kappa,\mathbf f}(\boldsymbol{\mu})$ is then well defined.
Furthermore, if Case~II takes place then, by (\ref{defR}), (\ref{Re}) and (\ref{wen}),
\begin{align}\label{C2}G_{\kappa,\mathbf f}(\boldsymbol{\mu})&=\|R\boldsymbol{\mu}\|^2_\kappa+2\sum_{i\in I}\,s_i\kappa(\zeta,\mu^i)\\{}&=\|R\boldsymbol{\mu}\|^2_\kappa+2\kappa(\zeta,R\boldsymbol{\mu})=
\|R\boldsymbol{\mu}+\zeta\|^2_\kappa-\|\zeta\|_\kappa^2\notag\end{align}
and consequently
\begin{equation}\label{GII}-\infty<-\|\zeta\|_\kappa^2\leqslant G_{\kappa,\mathbf f}(\boldsymbol{\mu})<\infty\text{ \ for all \ }\boldsymbol{\mu}\in\mathcal E^+_\kappa(\mathbf A).\end{equation}

Also fix a numerical vector $\mathbf a=(a_i)_{i\in I}$ with $a_i>0$ and a vector-valued function $\mathbf{g}=(g_i)_{i\in I}$ where all the $g_i: X\to(0,\infty)$ are continuous and such that
\begin{equation}\label{infg}g_{i,\inf}:=\inf_{x\in X}\,g_i(x)>0.\end{equation}
Write
\[\mathfrak M^+(\mathbf A,\mathbf a,\mathbf g):=\bigl\{\boldsymbol{\mu}\in\mathfrak M^+(\mathbf A): \ \langle g_i,\mu^i\rangle=a_i\quad\text{for all \ }i\in I\bigr\},\]
\[\mathcal E^+_\kappa(\mathbf A,\mathbf a,\mathbf g):=\mathcal E^+_\kappa(\mathbf A)\cap\mathfrak M^+(\mathbf A,\mathbf a,\mathbf g).\]
Because of (\ref{infg}), we thus have
\begin{equation}\label{up-est}\mu^i(A_i)\leqslant a_ig_{i,\inf}^{-1}<\infty\text{ \ for all \ }\boldsymbol{\mu}\in\mathfrak M^+(\mathbf A,\mathbf a,\mathbf g).\end{equation}
Since any $\psi\in\Psi(X)$ is lower bounded if $X$ is compact, and it is ${}\geqslant0$ otherwise, we conclude in Case I from (\ref{up-est}) that there is $M_{\mathbf f}\in(0,\infty)$ such that
\begin{equation}\label{GI}G_{\kappa,\mathbf f}(\boldsymbol{\mu})\geqslant-M_{\mathbf f}>-\infty\text{ \ for all \ }\boldsymbol{\mu}\in\mathcal E^+_\kappa(\mathbf A,\mathbf a,\mathbf g).\end{equation}

Also denote
\[\mathcal E^+_{\kappa,\mathbf f}(\mathbf A,\mathbf a,\mathbf g):=\mathcal E^+_{\kappa,\mathbf f}(\mathbf A)\cap\mathfrak M^+(\mathbf A,\mathbf a,\mathbf g),\]
\[G_{\kappa,\mathbf f}(\mathbf A,\mathbf a,\mathbf g):=\inf_{\boldsymbol{\mu}\in\mathcal E^+_{\kappa,\mathbf f}(\mathbf A,\mathbf a,\mathbf g)}\,G_{\kappa,\mathbf f}(\boldsymbol{\mu}).\]
In either Case I or Case II,  we then get from (\ref{GII}) and (\ref{GI})
\begin{equation}\label{Gfb}G_{\kappa,\mathbf f}(\mathbf A,\mathbf a,\mathbf g)>-\infty.\end{equation}

If the class $\mathcal E^+_{\kappa,\mathbf f}(\mathbf A,\mathbf a,\mathbf g)$ is nonempty, or equivalently if
\begin{equation}\label{Gf}G_{\kappa,\mathbf f}(\mathbf A,\mathbf a,\mathbf g)<\infty,\end{equation}
then the following (unconstrained) $\mathbf{f}$-weighted minimum energy problem, also known as the {\it Gauss variational problem\/} \cite{Gauss,O}, makes sense.

\begin{problem}\label{pr1} Does there exist $\boldsymbol{\lambda}_{\mathbf A}\in\mathcal E^+_{\kappa,\mathbf f}(\mathbf A,\mathbf a,\mathbf g)$ with $G_{\kappa,\mathbf f}(\boldsymbol{\lambda}_{\mathbf A})=G_{\kappa,\mathbf f}(\mathbf A,\mathbf a,\mathbf g)${\rm?}\end{problem}

If $I^+=\{1\}$, $I^-=\varnothing$, $\mathbf g=\mathbf1$, $\mathbf a=\mathbf1$ and $\mathbf f=\mathbf0$, then Problem~\ref{pr1} reduces to the minimum energy problem (\ref{cap-def}) solved by \cite[Theorem~4.1]{F1} (see Remark~\ref{remark} above).

\begin{remark} An analysis similar to that for a standard condenser, cf.\ \cite[Lemma~6.2]{ZPot2}, shows that requirement (\ref{Gf}) is fulfilled if and only if $c_\kappa(\dot{A}_i^\delta)>0$ for every $i\in I$, where
\begin{equation}\label{circ}\dot{A}_i^\delta:=\bigl\{x\in A^\delta_i: \ |f_i(x)|<\infty\bigr\},\end{equation}
$A^\delta_i$ being defined by (\ref{delta}). By (\ref{cnon0}), this yields that (\ref{Gf}) holds automatically whenever Case~II takes place, for the potential of $\zeta\in\mathcal E_\kappa(X)$ is finite n.e.\ on $X$ \cite[p.~164]{F1}.\end{remark}

\begin{remark}\label{r-2}If $\mathbf A$ is a compact standard condenser, the kernel $\kappa$ is continuous on $A^+\times A^-$, and Case~I holds, then the solvability  of Problem~\ref{pr1} can easily be established by exploiting the vague topology only, since then $\mathfrak M^+(\mathbf A,\mathbf a,\mathbf g)$ is vaguely compact, while $G_{\kappa,\mathbf f}(\cdot)$ is va\-guely l.s.c.\ on $\mathcal E^+_{\kappa,\mathbf f}(\mathbf A)$ (see \cite[Theorem~2.30]{O}).\footnote{If $\kappa$ is (finitely) continuous on $X\times X$, then this result by Ohtsuka can be extended to a compact generalized condenser. Such a generalization is established with the aid of Lemma~\ref{lemma-rel-comp} in a way similar to that in the proof of Theorem~\ref{th-main-comp} (see below). Since none of the classical kernels is continuous for $x=y$, we shall not go into detail.} However, these arguments break down if any of the $A_i$ is noncompact in $X$, for then $\mathfrak M^+(\mathbf A,\mathbf a,\mathbf g)$ is no longer vaguely compact.\end{remark}

The purpose of the example below is to give an explicit formula for a solution to Problem\/~{\rm\ref{pr1}} with a particular choice of $X$, $\kappa$, $\mathbf A$, $\mathbf a$, $\mathbf g$, and $\mathbf f$. Write $S_r:=S(0,r)$.

\begin{example}\label{zu}Let $\kappa_2(x,y)=|x-y|^{2-n}$ be the Newtonian kernel on $\mathbb R^n$ with $n\geqslant3$, $I^+=\{1\}$, $I^-=\{2\}$, $\mathbf g=\mathbf1$, $\mathbf a=\mathbf1$, $\mathbf f=\mathbf0$, $A_1=S_{r_1}$, and $A_2=S_{r_2}$, where $0<r_1<r_2<\infty$. According to Remark~\ref{r-2}, a solution to Problem~\ref{pr1} exists. Let $\lambda_r$, $0<r<\infty$, denote the $\kappa_2$-capacitary measure on $S_r$ (see Remark~\ref{remark} above); then by symmetry $\lambda_r$ is the uniformly distributed unit mass over $S_r$. Based on well known properties of the Newtonian potential of $\lambda_r$ \cite[Chapter~II, Section~3, n$^\circ$\,13]{L} we have
\begin{equation}\label{value1}c_{\kappa_2}\bigl(S_r\bigr)=r^{n-2},\end{equation}
$\kappa_2(x,\lambda_r)=r^{2-n}$ for all $x\in\overline{B}(0,r)$ and $\kappa_2(x,\lambda_r)=R^{2-n}$ for all $x\in S_R$, $R>r$.
Thus
\begin{equation}\label{val3}\kappa_2(\cdot,\lambda_{r_1}-\lambda_{r_2})=\left\{
\begin{array}{lll} r_1^{2-n}-r_2^{2-n} & \mbox{on} & S_{r_1},\\ 0 & \mbox{on} & S_{r_2}.\\ \end{array} \right. \end{equation}
Application of \cite[Proposition~1(iv)]{Z5}, providing characteristic properties of solutions to Problem~\ref{pr1} for a standard condenser, shows that $\boldsymbol\lambda:=(\lambda_{r_1},\lambda_{r_2})$ solves Problem~\ref{pr1} with $X$, $\kappa$, $\mathbf A$, $\mathbf a$, $\mathbf g$, and $\mathbf f$, chosen above. Hence the corresponding minimum value $G_{\kappa,\mathbf f}(\mathbf A,\mathbf a,\mathbf g)$ equals $\kappa_2(\boldsymbol\lambda,\boldsymbol\lambda)$ and\footnote{$\lambda_{r_1}$ is in fact the solution to the problem (\ref{cap-def}) for $S_{r_1}$ relative to the classical Green kernel $G$ on $B(0,r_2)$,
while $c_G\bigl(S_{r_1}\bigr)=\bigl[r_1^{2-n}-r_2^{2-n}\bigr]^{-1}$. This follows from (\ref{val3}) and (\ref{value}) by \cite[Lemmas~3.4, 3.5]{DFHSZ2}.\label{foot-zu}}
\begin{equation}\label{value}\kappa_2(\boldsymbol\lambda,\boldsymbol\lambda)=\|\lambda_{r_1}-\lambda_{r_2}\|^2_{\kappa_2}=r_1^{2-n}-r_2^{2-n}.\end{equation}
\end{example}

\begin{remark}\label{r-3}Assume that $\mathbf A$ is still a standard condenser, though now, in contrast to Remark~\ref{r-2}, its plates may be noncompact in $X$. Under the assumption (\ref{eq-intr}), an approach has been worked out in \cite{ZPot2,ZPot3}, based on both the vague and the strong topologies on $\mathcal E_\kappa^+(\mathbf A)$, which made it possible to provide a fairly complete analysis of Problem~\ref{pr1}. In more detail, it has been shown that if the kernel $\kappa$ is perfect and all the $g_i|_{A_i}$, $i\in I$, are bounded, then in either Case~I or Case~II the requirement
\begin{equation}\label{r-suff}c_\kappa(A)<\infty\end{equation}
is sufficient for Problem~\ref{pr1} to be solvable for every vector $\mathbf a$ \cite[Theorem~8.1]{ZPot2}. However, if (\ref{r-suff}) does not hold then in general there exists a vector $\mathbf a'$ such that the problem admits no solution \cite{ZPot2}.\footnote{In the case of the $\alpha$-Riesz kernels of order $1<\alpha\leqslant2$ on $\mathbb R^3$, some of the (theoretical) results on the solvability or unsolvability of Problem~\ref{pr1} obtained in \cite{ZPot2} have been illustrated in \cite{HWZ,OWZ} by means of numerical experiments.} Therefore, it was interesting to give a description of the set of all vectors $\mathbf a$ for which Problem~\ref{pr1} is nevertheless solvable. Such a characterization has been established in \cite{ZPot3}. See also footnote~\ref{foot} below.\end{remark}

Unless explicitly stated otherwise, in all that
follows we do not assume (\ref{eq-intr}) to hold.
Then the results obtained in \cite{ZPot2,ZPot3} and the approach developed are no longer valid.
In particular, assumption (\ref{r-suff}) does not guarantee anymore that $G_{\kappa,\mathbf f}(\mathbf A,\mathbf a,\mathbf g)$ is attained among $\boldsymbol{\mu}\in\mathcal E^+_{\kappa,\mathbf f}(\mathbf A,\mathbf a,\mathbf g)$. This can be illustrated by the following assertion.

\begin{figure}[tbp]
\begin{center}
 \includegraphics[width=4.in]{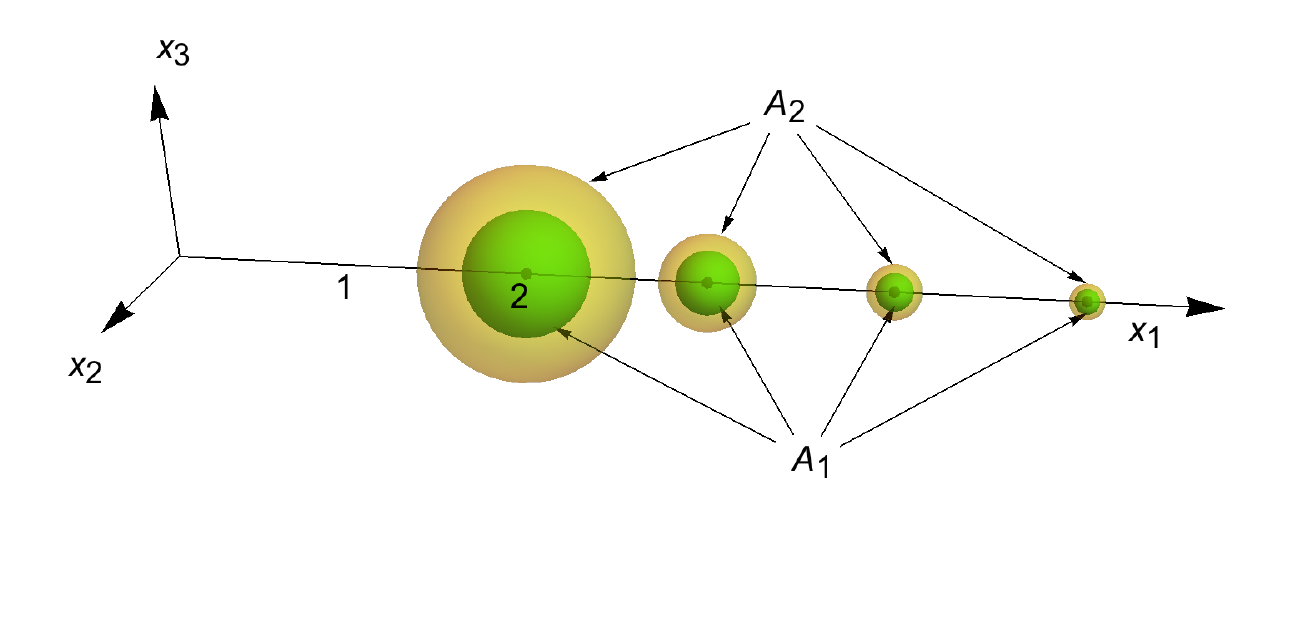}
\vspace{-.4in}
\caption{The condenser plates $A_1$ and $A_2$ from Theorem~\ref{pr:nons} with $n=3$.}
\label{FigThm4.6}
\end{center}
\end{figure}

\begin{theorem}\label{pr:nons}Let\/ $X=\mathbb R^n$ with\/ $n\geqslant3$, $I^+=\{1\}$, $I^-=\{2\}$, $\mathbf g=\mathbf 1$, $\mathbf a=\mathbf 1$,
$\mathbf f=\mathbf 0$,
\[A_1=\bigcup_{k\geqslant2}\,S(x_k,r_{1,k}),\quad A_2=\bigcup_{k\geqslant2}\,S(x_k,r_{2,k}),\]
where\/ $x_k=(k,0,\ldots,0)$, $r_{2,k}^{2-n}=k^2$ and\/ $r_{1,k}^{2-n}=k^2+k^{-q}$ with\/ $q\in(0,\infty)$, and let\/ $\kappa=\kappa_2$ be the Newtonian kernel. Then\/ $G_{\kappa,\mathbf f}(\mathbf A,\mathbf a,\mathbf g)$ equals\/ $0$ and hence cannot be an actual minimum.\end{theorem}

\begin{proof}Note that $\mathbf A=(A_1,A_2)$ forms a standard condenser in $\mathbb R^n$ such that (\ref{eq-intr}) fails to hold. Let $\lambda_{k,r}$, $0<r<\infty$, denote the $\kappa_2$-capacitary measure on $S(x_k,r)$ (see Remark~\ref{remark}). Then $\boldsymbol\lambda_k=(\lambda_{k,r_{1,k}},\lambda_{k,r_{2,k}})$, $k\in\mathbb N$, are admissible for Problem~\ref{pr1}, and therefore, by (\ref{Re}),
\[0\leqslant G_{\kappa_2,\mathbf f}(\mathbf A,\mathbf a,\mathbf g)\leqslant\kappa_2(\boldsymbol\lambda_k,\boldsymbol\lambda_k)=\|\lambda_{k,r_{1,k}}-\lambda_{k,r_{2,k}}\|^2_{\kappa_2}, \ k\in\mathbb N.\]
According to (\ref{value}), the right-hand side equals
\[r_{1,k}^{2-n}-r_{2,k}^{2-n}=k^{-q},\]
and hence it tends to $0$ as $k\to\infty$. This yields that $G_{\kappa_2,\mathbf f}(\mathbf A,\mathbf a,\mathbf g)=0$. By the strict positive definiteness of $\kappa_2$, $G_{\kappa_2,\mathbf f}(\mathbf A,\mathbf a,\mathbf g)$ cannot therefore be an actual minimum, though $c_{\kappa_2}(A)<\infty$, which is clear from (\ref{value1}) by the countable subadditivity of the inner capacity on the universally measurable sets \cite[Lemma~2.3.5]{F1}.\end{proof}

Using the electrostatic interpretation, which is possible for the Coulomb kernel $|x-y|^{-1}$ on $\mathbb R^3$, we say that a short-circuit occurs between the oppositely charged plates $A_1$ and $A_2$ from Theorem~\ref{pr:nons}, which touch each other at the point at infinity. This certainly may also happen for a generalized condenser (see Definition~\ref{def-cond-st}).
Therefore, it is meaningful to ask what kinds of additional requirements on the vector measures under consideration will prevent this phenomenon, and secure that a solution to the corresponding $\mathbf{f}$-weighted minimum energy problem does exist.
The idea below is to impose such an upper constraint on the measures from $\mathfrak M^+(\mathbf A,\mathbf a,\mathbf g)$ which would prevent the blow-up effect.

\section{Constrained $\mathbf{f}$-weighted minimum energy problem}\label{sec-Gauss-c}

\subsection{Statement of the problem} Unless stated otherwise, in all that follows $\kappa$, $\mathbf A$, $\mathbf a$, $\mathbf g$, and $\mathbf f$ are as indicated at the beginning of the preceding section. Let $\mathfrak C(\mathbf A)$ consist of all $\boldsymbol{\sigma}=(\sigma^i)_{i\in I}\in\mathfrak M^+(\mathbf A)$ such that $\langle g_i,\sigma^i\rangle>a_i$ and\footnote{Recall that $S(\mu)=S^\mu_X$ denotes the support of $\mu\in\mathfrak M(X)$. For the notation $\langle g_i,\sigma^i\rangle$  see (\ref{ul}), noting that the $g_i$, $i\in I$, are continuous on $X$ and~${}>0$.}
\begin{equation}\label{g-mass}S^{\sigma^i}_{X}=A_i\text{ \ for all \ }i\in I.\end{equation}
These $\boldsymbol{\sigma}$ will serve as {\it constraints\/} for $\boldsymbol{\mu}\in\mathfrak M^+(\mathbf A,\mathbf a,\mathbf g)$.
Given $\boldsymbol{\sigma}\in\mathfrak C(\mathbf A)$, write
\[\mathfrak M^{\boldsymbol{\sigma}}(\mathbf A):=\bigl\{\boldsymbol{\mu}\in\mathfrak M^+(\mathbf A): \ \mu^i\leqslant\sigma^i\quad\text{for all \ }i\in I\bigr\},\]
where $\mu^i\leqslant\sigma^i$ means that $\sigma^i-\mu^i$ is a positive scalar Radon measure on $X$, and
\begin{align*}
\mathfrak M^{\boldsymbol{\sigma}}(\mathbf A,\mathbf a,\mathbf g)&:=\mathfrak M^{\boldsymbol{\sigma}}(\mathbf A)\cap\mathfrak M^+(\mathbf A,\mathbf a,\mathbf g),\\
\mathcal E^{\boldsymbol{\sigma}}_{\kappa,\mathbf f}(\mathbf A,\mathbf a,\mathbf g)&:=\mathcal E^+_{\kappa,\mathbf f}(\mathbf A)\cap\mathfrak M^{\boldsymbol{\sigma}}(\mathbf A,\mathbf a,\mathbf g).\end{align*}
Note that $\mathfrak M^{\boldsymbol{\sigma}}(\mathbf A)$ along with $\mathfrak M^+(\mathbf A)$ is vaguely closed, for so is $\mathfrak M^+(X)$.

Since $\mathcal E^{\boldsymbol{\sigma}}_{\kappa,\mathbf f}(\mathbf A,\mathbf a,\mathbf g)\subset\mathcal E^+_{\kappa,\mathbf f}(\mathbf A,\mathbf a,\mathbf g)$, we get from (\ref{Gfb})
\begin{equation}\label{-finite}-\infty<G_{\kappa,\mathbf f}(\mathbf A,\mathbf a,\mathbf g)\leqslant G^{\boldsymbol{\sigma}}_{\kappa,\mathbf f}(\mathbf A,\mathbf a,\mathbf g):=\inf_{\boldsymbol{\mu}\in\mathcal E^{\boldsymbol{\sigma}}_{\kappa,\mathbf f}(\mathbf A,\mathbf a,\mathbf g)}\,G_{\kappa,\mathbf f}(\boldsymbol{\mu})\leqslant\infty.\end{equation}
Unless explicitly stated otherwise, in all that
follows we assume that the class $\mathcal E^{\boldsymbol{\sigma}}_{\kappa,\mathbf f}(\mathbf A,\mathbf a,\mathbf g)$ is nonempty, or equivalently
\begin{equation}\label{gauss-finite}G_{\kappa,\mathbf f}^{\boldsymbol{\sigma}}(\mathbf A,\mathbf a,\mathbf g)<\infty.\end{equation}
Then the following constrained $\mathbf{f}$-weighted minimum energy problem, also known as the {\it constrained Gauss variational problem\/}, makes sense.

\begin{problem}\label{pr2}Given $\boldsymbol{\sigma}\in\mathfrak C(\mathbf A)$, does there exist $\boldsymbol{\lambda}_{\mathbf A}^{\boldsymbol{\sigma}}\in\mathcal E^{\boldsymbol{\sigma}}_{\kappa,\mathbf f}(\mathbf A,\mathbf a,\mathbf g)$ with \[G_{\kappa,\mathbf f}(\boldsymbol{\lambda}^{\boldsymbol{\sigma}}_{\mathbf A})=G_{\kappa,\mathbf f}^{\boldsymbol{\sigma}}(\mathbf A,\mathbf a,\mathbf g){\rm?}\]\end{problem}

Note that assumption (\ref{g-mass}) in fact causes no restriction on the objects in question since, if it does not hold, then Problem~\ref{pr2} reduces to the same problem for the generalized condenser $\mathbf F:=\bigl(S_X^{\sigma^i}\bigr)_{i\in I}$ in place of $\mathbf A=(A_i)_{i\in I}$, because $\mathcal E^{\boldsymbol{\sigma}}_{\kappa,\mathbf f}(\mathbf F,\mathbf a,\mathbf g)=\mathcal E^{\boldsymbol{\sigma}}_{\kappa,\mathbf f}(\mathbf A,\mathbf a,\mathbf g)$.

\begin{lemma}\label{suff-fin}For\/ {\rm(\ref{gauss-finite})} to hold, it is sufficient that for every\/ $i\in I$, $\langle g_i,\sigma^i|_{\dot{A}_i^\delta}\rangle>a_i$ where\/ $\dot{A}_i^\delta$ is defined by\/ {\rm(\ref{circ})}, and in addition $\kappa(\sigma^i|_K,\sigma^i|_K)<\infty$ for any compact\/ $K\subset\dot{A}_i^\delta$.
\end{lemma}

\begin{proof}Noting that the $\dot{A}_i^\delta$, $i\in I$, are universally measurable, we see that for every $i\in I$ there exists a compact set $K_i\subset\dot{A}_i^\delta$ such that $\langle g_i,\sigma^i|_{K_i}\rangle>a_i$ and $|f_i|\leqslant M_i<\infty$ on $K_i$. Then $\boldsymbol\mu\in\mathcal E^{\boldsymbol\sigma}_{\kappa,\mathbf f}(\mathbf A,\mathbf a,\mathbf g)$ with $\mu^i:=a_i\sigma^i|_{K_i}/\langle g_i,\sigma^i|_{K_i}\rangle$, $i\in I$, and (\ref{gauss-finite}) follows.\end{proof}

\begin{remark}\label{rem-3}Assume for a moment that (\ref{eq-intr}) holds. It has been shown by \cite[Theorem~6.2]{Z9} that if, in addition, the kernel $\kappa$ is perfect, all the
$g_i|_{A_i}$, $i\in I$, are bounded, and condition (\ref{r-suff}) is satisfied, then in both Cases~I and~II Problem~\ref{pr2} is solvable for any vector $\mathbf a$ and any constraint $\boldsymbol{\sigma}\in\mathfrak C(\mathbf A)$.\footnote{Actually, the results described in Remarks~\ref{r-3} and \ref{rem-3} have been obtained in \mbox{\cite{Z9}--\cite{ZPot3}} even for {\it infinite\/} dimensional vector measures.\label{foot}} But if requirement (\ref{eq-intr}) is omitted then the approach developed in \cite{Z9} breaks down, and the quoted result is no longer valid. This can be seen for $\kappa_2$ on $\mathbb R^n$, $n\geqslant3$, if we restrict ourselves to $\boldsymbol\mu\in\mathcal E^{\boldsymbol\sigma}_{\kappa_2,\mathbf f}(\mathbf A,\mathbf a,\mathbf g)$ with $\sigma^i:=\sum_{k\in\mathbb N}\,\lambda_{k,r_{i,k}}$, $i=1,2$, where $\mathbf A$, $\mathbf a$, $\mathbf g$, $\mathbf f$ and $\lambda_{k,r_{i,k}}$ are as chosen in Theorem~\ref{pr:nons} and its proof. Observe that these $\sigma^i$, $i=1,2$, are {\it unbounded\/}; compare with Theorem~\ref{th-main} below.\end{remark}

Let $\mathfrak S^{\boldsymbol{\sigma}}_{\kappa,\mathbf f}(\mathbf A,\mathbf a,\mathbf g)$ (possibly empty) consist of all the solutions to Problem~\ref{pr2}.

\begin{lemma}\label{lemma:unique:}
If\/ $\boldsymbol{\lambda}$ and\/ $\widehat{\boldsymbol{\lambda}}$ are two elements of\/ $\mathfrak S^{\boldsymbol{\sigma}}_{\kappa,\mathbf f}(\mathbf A,\mathbf a,\mathbf g)$ then\/
$\|\boldsymbol{\lambda}-\widehat{\boldsymbol{\lambda}}\|_{\mathcal
E^+_\kappa(\mathbf{A})}=0$, and hence\/ $\boldsymbol{\lambda}$ and\/ $\widehat{\boldsymbol{\lambda}}$ are\/ $R$-equivalent in\/ $\mathcal E_\kappa^+(\mathbf A)$. Thus a solution to Problem\/~{\rm{\ref{pr2}}} is unique\/ {\rm(}provided it exists\/{\rm)} whenever all the\/ $A_i$, $i\in I$, are mutually essentially disjoint.
\end{lemma}

\begin{proof} Since the class $\mathcal E^{\boldsymbol{\sigma}}_{\kappa,\mathbf f}(\mathbf A,\mathbf a,\mathbf g)$ is convex, we conclude from (\ref{wen}) and (\ref{Re})
that
\[4G_{\kappa,\mathbf f}^{\boldsymbol{\sigma}}(\mathbf A,\mathbf a,\mathbf g)\leqslant
4G_{\kappa,\mathbf f}\Bigl(\frac{\boldsymbol{\lambda}+\widehat{\boldsymbol{\lambda}}}{2}\Bigr)=
\|R\boldsymbol{\lambda}+R\widehat{\boldsymbol{\lambda}}\|_\kappa^2+
4\langle\mathbf{f},\boldsymbol{\lambda}+\widehat{\boldsymbol{\lambda}}\rangle.\]
On the other hand, applying the parallelogram identity in the
pre-Hilbert space $\mathcal E_\kappa(X)$ to $R\boldsymbol{\lambda}$
and $R\widehat{\boldsymbol{\lambda}}$ and then adding and
subtracting
$4\langle\mathbf{f},\boldsymbol{\lambda}+\widehat{\boldsymbol{\lambda}}\rangle$
we get
\[\|R\boldsymbol{\lambda}-R\widehat{\boldsymbol{\lambda}}\|_\kappa^2=
-\|R\boldsymbol{\lambda}+R\widehat{\boldsymbol{\lambda}}\|_\kappa^2-4\langle\mathbf{f},\boldsymbol{\lambda}+
\widehat{\boldsymbol{\lambda}}\rangle+2G_{\kappa,\mathbf{f}}(\boldsymbol{\lambda})+
2G_{\kappa,\mathbf{f}}(\widehat{\boldsymbol{\lambda}}).\] When combined
with the preceding relation, this yields
\[0\leqslant\|R\boldsymbol{\lambda}-R\widehat{\boldsymbol{\lambda}}\|^2_\kappa\leqslant-
4G_{\kappa,\mathbf{f}}^{\boldsymbol{\sigma}}(\mathbf{A},\mathbf{a},\mathbf{g})+2G_{\kappa,\mathbf{f}}(\boldsymbol{\lambda})+
2G_{\kappa,\mathbf{f}}(\widehat{\boldsymbol{\lambda}})=0,\] which
establishes the former assertion of the lemma because of (\ref{isom}). On account of Lemma~\ref{inj}, this completes the proof.\end{proof}

The following example shows that in general a solution to Problem~\ref{pr2} is not unique.

\begin{example}Let $\kappa_2$ be the Newtonian kernel on $\mathbb R^n$ with $n\geqslant3$, $I=I^+=\{1,2\}$, $\mathbf a=\mathbf 1$, $\mathbf f=\mathbf 0$, $\mathbf g=\mathbf 1$, and let $A_1=A_2$ be the $(n-1)$-dimensional unit sphere $S:=S(0,1)$. Let $\lambda$ denote the $\kappa_2$-capacitary measure on $S$; then $S(\lambda)=S$. Choose $\sigma^1=\sigma^2=3\lambda$. Then $\boldsymbol\lambda=(\lambda,\lambda)$ is obviously a solution to Problem~5.1. Choose compact disjoint sets $K_k\subset S$, $k=1,2$, so that $c_2(K_1)=c_2(K_2)>0$ and define $\lambda_k=\lambda|_{K_k}/2$. Then $\boldsymbol\lambda_k=(\lambda-\lambda_k,\lambda+\lambda_k)$, $k=1,2$, are admissible measures for Problem~\ref{pr2} with the data chosen above such that $R\boldsymbol\lambda_k=R\boldsymbol\lambda$, $k=1,2$, and hence $\kappa_2(\boldsymbol\lambda_k,\boldsymbol\lambda_k)=\kappa_2(\boldsymbol\lambda,\boldsymbol\lambda)$, $k=1,2$. Thus each of the non-equal, although $R$-equivalent vector measures $\boldsymbol\lambda$, $\boldsymbol\lambda_1$, and $\boldsymbol\lambda_2$ is a solution to Problem~\ref{pr2}.\end{example}

\subsection{Auxiliary results} In view on the definition of the vague topology on $\mathfrak M^+(\mathbf A)$, we call a set $\mathfrak F\subset\mathfrak M^+(\mathbf A)$ {\it
vaguely bounded\/} if for every $i\in I$ and every $\varphi\in C_0(X)$,
\[\sup_{\boldsymbol{\mu}\in\mathfrak F}\,|\mu^i(\varphi)|<\infty.\]

\begin{lemma}\label{lem:vaguecomp} If\/ $\mathfrak F\subset\mathfrak
M^+(\mathbf A)$ is vaguely bounded, then it is vaguely relatively
compact.\end{lemma}

\begin{proof} It is clear from the above definition that for every $i\in I$ the set
\[\mathfrak F^i:=\bigl\{\mu^i\in\mathfrak M^+(A_i):  \ \boldsymbol\mu=(\mu^j)_{j\in I}\in\mathfrak F\bigr\}\]
is vaguely bounded in $\mathfrak M^+(X)$, and hence $\mathfrak F^i$ is vaguely relatively compact in $\mathfrak M(X)$ \cite[Chapter~III, Section~2, Proposition~9]{B2}.
Since $\mathfrak F\subset\prod_{i\in I}\,\mathfrak F^i$, Tychonoff's theorem on the product of compact spaces \cite[Chapter~I, Section~9, Theorem~3]{B1} implies the lemma.\end{proof}

Let $\mathfrak M^+(\mathbf A,\leqslant\!\mathbf a,\mathbf g)$ consist of all $\boldsymbol{\mu}\in\mathfrak M^+(\mathbf A)$ with $\langle g_i,\mu^i\rangle\leqslant a_i$ for all $i\in I$. We also write $\mathfrak M^{\boldsymbol{\sigma}}(\mathbf A,\leqslant\!\mathbf a,\mathbf g):=\mathfrak M^{\boldsymbol{\sigma}}(\mathbf A)\cap\mathfrak M^+(\mathbf A,\leqslant\!\mathbf a,\mathbf g)$. By (\ref{infg}),
\begin{equation}\label{bbbound}\mu^i(A_i)\leqslant a_ig_{i,\inf}^{-1}<\infty\text{ \ for all \ }\boldsymbol{\mu}\in\mathfrak M^+(\mathbf A,\leqslant\!\mathbf a,\mathbf g).\end{equation}

\begin{lemma}\label{lemma-rel-comp} $\mathfrak M^+(\mathbf A,\leqslant\!\mathbf a,\mathbf g)$ and\/ $\mathfrak M^{\boldsymbol{\sigma}}(\mathbf A,\leqslant\!\mathbf a,\mathbf g)$ are vaguely bounded and closed, and hence they both are vaguely compact. If the condenser\/ $\mathbf A$ is compact then the same holds for\/ $\mathfrak M^+(\mathbf A,\mathbf a,\mathbf g)$ and\/ $\mathfrak M^{\boldsymbol{\sigma}}(\mathbf A,\mathbf a,\mathbf g)$.\end{lemma}

\begin{proof} It is seen from (\ref{bbbound}) that $\mathfrak M^+(\mathbf A,\leqslant\!\mathbf a,\mathbf g)$ is vaguely bounded. Hence, by Lem\-ma~\ref{lem:vaguecomp}, for any net $(\boldsymbol{\mu}_s)_{s\in S}\subset\mathfrak M^+(\mathbf A,\leqslant\!\mathbf a,\mathbf g)$ there exists a vague cluster point $\boldsymbol{\mu}$. Since $\mathfrak M^+(\mathbf A)$ is vaguely closed, we have $\boldsymbol{\mu}\in\mathfrak M^+(\mathbf A)$. Choose a subnet  $(\boldsymbol{\mu}_{t})_{t\in T}$ of $(\boldsymbol{\mu}_s)_{s\in S}$ converging vaguely to $\boldsymbol{\mu}$. As $g_i$ is positive and continuous, we get from Lemma~\ref{lemma-semi}
\begin{equation}\label{rel-lemma-rel-comp}\langle g_i,\mu^i\rangle\leqslant\liminf_{t\in T}\,\langle g_i,\mu_{t}^i\rangle\leqslant a_i\quad\text{for all \ }i\in I.\end{equation}
Thus $\boldsymbol{\mu}\in\mathfrak M^+(\mathbf A,\leqslant\!\mathbf a,\mathbf g)$, which shows that indeed $\mathfrak M^+(\mathbf A,\leqslant\!\mathbf a,\mathbf g)$ is vaguely closed and compact. Since $\mathfrak M^{\boldsymbol{\sigma}}(\mathbf A)$ is vaguely closed in $\mathfrak M^+(\mathbf A)$, the first assertion of the lemma follows.
Assume now that $\mathbf A$ is compact, and let the above $(\boldsymbol{\mu}_s)_{s\in S}$ be taken from $\mathfrak M^+(\mathbf A,\mathbf a,\mathbf g)$. Then $\nu\mapsto\langle g_i,\nu\rangle$ is vaguely continuous on $\mathfrak M^+(A_i)$ for every $i\in I$, and therefore all the inequalities in (\ref{rel-lemma-rel-comp}) are in fact equalities. Thus $\mathfrak M^+(\mathbf A,\mathbf a,\mathbf g)$ and hence also $\mathfrak M^{\boldsymbol{\sigma}}(\mathbf A,\mathbf a,\mathbf g)$ are vaguely closed and  compact.\end{proof}

\begin{lemma}\label{lemma-rel-cl} $\mathfrak M^{\boldsymbol{\sigma}}(\mathbf A,\mathbf a,\mathbf g)$ is vaguely compact for any\/ $\boldsymbol{\sigma}\in\mathfrak C(\mathbf A)$ possessing the property\footnote{For a compact $\mathbf A$ relation (\ref{boundd}) holds automatically, and Lemma~\ref{lemma-rel-cl} then in fact reduces to Lemma~\ref{lemma-rel-comp}.}
\begin{equation}\label{boundd}\langle g_i,\sigma^i\rangle<\infty\quad\text{for all \ }i\in I.\end{equation}
\end{lemma}

\begin{proof} Fix a vague cluster point $\boldsymbol{\mu}$ of a net $(\boldsymbol{\mu}_s)_{s\in S}\subset\mathfrak M^{\boldsymbol{\sigma}}(\mathbf A,\mathbf a,\mathbf g)$. By Lemma~\ref{lemma-rel-comp}, such a $\boldsymbol{\mu}$ exists and belongs to $\mathfrak M^{\boldsymbol{\sigma}}(\mathbf A,\leqslant\!\mathbf a,\mathbf g)$. We only need to show that, under requirement (\ref{boundd}), $\langle g_i,\mu^i\rangle=a_i$ for every $i\in I$. Passing to a subnet and changing notations assume that $\boldsymbol{\mu}$ is the vague limit of $(\boldsymbol{\mu}_s)_{s\in S}$.
Consider an exhaustion of $A_i$ by an upper directed family of compact sets $K\subset A_i$.\footnote{A family $\mathfrak Q$ of sets $Q\subset X$ is said to be {\it upper directed\/} if for any $Q_1,Q_2\in\mathfrak Q$ there exists $Q_3\in\mathfrak Q$ such that $Q_1\cup Q_2\subset Q_3$.}
Since the indicator function $1_K$ of $K$ is upper semicontinuous, we get from Lemma~\ref{lemma-semi} (with $\psi=-g_i1_{K}\in\Psi(X)$) and \cite[Lemma~1.2.2]{F1}
\begin{align*}a_i&\geqslant\langle g_i,\mu^i\rangle=\lim_{K\uparrow A_i}\,\langle g_i1_{K},\mu^i\rangle\geqslant\lim_{K\uparrow A_i}\,\limsup_{s\in S}\,\langle g_i1_{K},\mu_{s}^i\rangle\\&{}
=a_i-\lim_{K\uparrow A_i}\,\liminf_{s\in S}\,\langle g_i1_{A_i\setminus K},\mu_{s}^i\rangle.\end{align*}
Hence, the lemma will follow once we show that
\begin{equation}\label{g0}\lim_{K\uparrow A_i}\,\liminf_{s\in S}\,\langle g_i1_{A_i\setminus K},\mu_{s}^i\rangle=0.\end{equation}
Since, by (\ref{boundd}),
\[\infty>\langle g_i,\sigma^i\rangle=\lim_{K\uparrow A_i}\,\langle g_i1_{K},\sigma^i\rangle,\]
we have
\[\lim_{K\uparrow A_i}\,\langle g_i1_{A_i\setminus K},\sigma^i\rangle=0.\]
When combined with
\[\langle g_i1_{A_i\setminus K},\mu_{s}^i\rangle\leqslant\langle g_i1_{A_i\setminus K},\sigma^i\rangle\quad\text{for every \ }s\in S,\]
this implies (\ref{g0}) as desired.\end{proof}

\begin{lemma}\label{gauss:lsc} In Case\/~{\rm I} the mapping\/ $\boldsymbol\mu\mapsto G_{\kappa,\mathbf f}(\boldsymbol{\mu})$ is vaguely l.s.c.\ on\/ $\mathcal E_{\kappa,\mathbf f}^+(\mathbf A)$, and it is strongly continuous if Case\/~{\rm II} holds.\end{lemma}

\begin{proof}This is obtained from Lemma~\ref{lemma-semi} and relation (\ref{C2}), respectively.\end{proof}

\begin{definition}A net $(\boldsymbol\mu_s)_{s\in S}\subset\mathcal E^{\boldsymbol{\sigma}}_{\kappa,\mathbf f}(\mathbf A,\mathbf a,\mathbf g)$ is said to be {\it minimizing\/} in Problem~\ref{pr2} if
\begin{equation}\label{min-seq:}\lim_{s\in S}\,G_{\kappa,\mathbf f}(\boldsymbol\mu_s)=G_{\kappa,\mathbf f}^{\boldsymbol{\sigma}}(\mathbf A,\mathbf a,\mathbf g).\end{equation}
Let $\mathbb M^{\boldsymbol{\sigma}}_{\kappa,\mathbf f}(\mathbf A,\mathbf a,\mathbf g)$ consist of all these nets $(\boldsymbol\mu_s)_{s\in S}$; it is nonempty because of (\ref{gauss-finite}).
\end{definition}

\begin{lemma}For any\/ $(\boldsymbol\mu_s)_{s\in S}$ and\/ $(\boldsymbol\nu_t)_{t\in T}$ in\/ $\mathbb M^{\boldsymbol{\sigma}}_{\kappa,\mathbf f}(\mathbf A,\mathbf a,\mathbf g)$ we have
\begin{equation}
\lim_{(s,t)\in S\times T}\,\|\boldsymbol{\mu}_s-\boldsymbol{\nu}_t\|_{\mathcal E^+_\kappa(\mathbf A)}=0,\label{fund:}
\end{equation}
where\/ $S\times T$ is the upper directed product\footnote{See e.g.\ \cite[Chapter~2, Section~3]{K}.} of the upper directed sets\/ $S$
and\/~$T$.
\end{lemma}

\begin{proof} In the same manner as in the proof of Lemma~\ref{lemma:unique:} we get
\begin{equation*}0\leqslant\|R\boldsymbol{\mu}_s-R\boldsymbol{\nu}_t\|^2_\kappa\leqslant-
4G^{\boldsymbol{\sigma}}_{\kappa,\mathbf f}(\mathbf A,\mathbf a,\mathbf g)+2G_{\kappa,\mathbf f}(\boldsymbol{\mu}_s)+2G_{\kappa,\mathbf f}(\boldsymbol{\nu}_t),\end{equation*}
which gives (\ref{fund:}) when combined with (\ref{isom}), (\ref{min-seq:}) and the finiteness of $G_{\kappa,\mathbf f}^{\boldsymbol{\sigma}}(\mathbf A,\mathbf a,\mathbf g)$.\end{proof}

Taking here $(\boldsymbol\mu_s)_{s\in S}$ and $(\boldsymbol\nu_t)_{t\in T}$ to be equal, we  arrive at the following conclusion.

\begin{corollary}\label{cor:fund}Every\/ $(\boldsymbol\mu_s)_{s\in S}\in\mathbb M^{\boldsymbol{\sigma}}_{\kappa,\mathbf f}(\mathbf A,\mathbf a,\mathbf g)$ is strong Cauchy in\/ $\mathcal E^+_\kappa(\mathbf A)$.\end{corollary}

The result below will be used in subsequent work of the authors.

\begin{theorem}\label{ex-gnr}Let the kernel\/ $\kappa$ be perfect and let either\/ $I^+$ or\/ $I^-$ be empty. If moreover\/ {\rm(\ref{boundd})} holds, then in both Cases\/~{\rm I} and\/~{\rm II} Problem\/~{\rm\ref{pr2}} is solvable for any vector\/ $\mathbf a$. Furthermore, every\/ $(\mu_s)_{s\in S}\in\mathbb M^{\boldsymbol{\sigma}}_{\kappa,\mathbf f}(\mathbf A,\mathbf a,\mathbf g)$ converges to any\/ $\boldsymbol\lambda\in\mathfrak S^{\boldsymbol{\sigma}}_{\kappa,\mathbf f}(\mathbf A,\mathbf a,\mathbf g)$ strongly in\/ $\mathcal E_\kappa^+(\mathbf A)$; and hence also vaguely whenever the $A_i$, $i\in I$, are mutually essentially disjoint.\end{theorem}

\begin{proof}Assume for definiteness that $I^-=\varnothing$ and fix $(\boldsymbol\mu_s)_{s\in S}\in\mathbb M^{\boldsymbol{\sigma}}_{\kappa,\mathbf f}(\mathbf A,\mathbf a,\mathbf g)$. By Lemma~\ref{lemma-rel-cl}, there is a subnet $(\boldsymbol\mu_t)_{t\in T}$ of $(\boldsymbol\mu_s)_{s\in S}$ converging vaguely to some $\boldsymbol\mu\in\mathfrak M^{\boldsymbol{\sigma}}(\mathbf A,\mathbf a,\mathbf g)$.
The net $(\boldsymbol\mu_t)_{t\in T}$ belongs to $\mathbb M^{\boldsymbol{\sigma}}_{\kappa,\mathbf f}(\mathbf A,\mathbf a,\mathbf g)$, for so does $(\boldsymbol\mu_s)_{s\in S}$, and hence it is strong Cauchy in the semimetric space $\mathcal E^+_\kappa(\mathbf A)$ by Corollary~\ref{cor:fund}. Application of relations (\ref{Re}) and (\ref{isom}) then shows that the net of positive scalar measures $R\boldsymbol\mu_t$, $t\in T$, is strong Cauchy in the metric space $\mathcal E_\kappa^+(A)$, and therefore it is strongly bounded. Furthermore, since $R\boldsymbol\mu_t\to R\boldsymbol\mu$ vaguely in $\mathfrak M^+(X)$, we have $R\boldsymbol\mu_t\otimes R\boldsymbol\mu_t\to R\boldsymbol\mu\otimes R\boldsymbol\mu$ vaguely in $\mathfrak M^+(X\times X)$ \cite[Chapter~3, Section~5, Exercise~5]{B2}. Applying
Lemma~\ref{lemma-semi} to $X\times X$ and $\psi=\kappa$, we thus get $R\boldsymbol\mu\in\mathcal E^+_\kappa(A)$. Moreover, $(R\boldsymbol\mu_t)_{t\in T}$ converges to $R\boldsymbol\mu$ strongly in $\mathcal E_\kappa^+(X)$, which follows from the above in view of the perfectness of $\kappa$. Hence, by (\ref{isom}), $(\boldsymbol\mu_t)_{t\in T}$ converges to $\boldsymbol\mu$ strongly in $\mathcal E_\kappa^+(\mathbf A)$.

In either Case~I or Case~II, we get from (\ref{min-seq:}) and Lemma~\ref{gauss:lsc}
\begin{equation}\label{di-ag}-\infty<G_{\kappa,\mathbf f}(\boldsymbol\mu)\leqslant\lim_{t\in T}\,G_{\kappa,\mathbf f}(\boldsymbol\mu_t)=G_{\kappa,\mathbf f}^{\boldsymbol\sigma}(\mathbf A,\mathbf a,\mathbf g)<\infty,\end{equation}
where the first inequality is valid by (\ref{GII}) and (\ref{GI}), while the last one holds by (\ref{gauss-finite}).
This yields that $\boldsymbol\mu\in\mathcal E_{\kappa,\mathbf f}^{\boldsymbol{\sigma}}(\mathbf A,\mathbf a,\mathbf g)$, and therefore $G_{\kappa,\mathbf f}(\boldsymbol\mu)\geqslant G_{\kappa,\mathbf f}^{\boldsymbol\sigma}(\mathbf A,\mathbf a,\mathbf g)$. It is seen from (\ref{di-ag}) that in fact equality prevails in the last relation, and hence indeed
$\boldsymbol\mu\in\mathfrak S^{\boldsymbol{\sigma}}_{\kappa,\mathbf f}(\mathbf A,\mathbf a,\mathbf g)$.

Since a strong Cauchy net converges strongly to any of its strong cluster points (even in the present case of a semimetric space), it follows from the above that $\boldsymbol\mu_s\to\boldsymbol\mu$ strongly in $\mathcal E_\kappa^+(\mathbf A)$. Finally, if $\boldsymbol\lambda$ is any other element of the class $\mathfrak S^{\boldsymbol{\sigma}}_{\kappa,\mathbf f}(\mathbf A,\mathbf a,\mathbf g)$, then also $\boldsymbol\mu_s\to\boldsymbol\lambda$ strongly in $\mathcal E_\kappa^+(\mathbf A)$ because $\|\boldsymbol\mu-\boldsymbol\lambda\|_{\mathcal E_\kappa^+(\mathbf A)}=0$ according to Lemma~\ref{lemma:unique:}.

It has been shown that any of the vague cluster points of $(\boldsymbol\mu_s)_{s\in S}$ belongs to $\mathfrak S^{\boldsymbol{\sigma}}_{\kappa,\mathbf f}(\mathbf A,\mathbf a,\mathbf g)$. Assuming now that the $A_i$, $i\in I$, are mutually essentially disjoint, we see from the latter assertion of Lemma~\ref{lemma:unique:} that then the vague cluster set of $(\boldsymbol\mu_s)_{s\in S}$ reduces to the measure $\boldsymbol\mu$, chosen at the beginning of the proof. As $\mathfrak M^+(\mathbf A)$ is Hausdorff in the vague topology, $\boldsymbol\mu_s\to\boldsymbol\mu$ also vaguely \cite[Chapter~I, Section~9, n$^\circ$\,1, Corollary]{B1}.
\end{proof}

\section{On the solvability of Problem~\ref{pr2} for Riesz kernels.~I}\label{sec:solvI}

Throughout Sections~\ref{sec:solvI}--\ref{sec:dual}, let $n\geqslant3$, $n\in\mathbb N$, and $\alpha\in(0,n)$ be fixed. On $X=\mathbb R^n$, consider the $\alpha$-{\it Riesz kernel\/} $\kappa(x,y)=\kappa_\alpha(x,y):=|x-y|^{\alpha-n}$ of order $\alpha$.
The $\alpha$-Riesz kernel is known to be strictly positive definite and moreover perfect \cite{D1,D2}, and hence the metric space $\mathcal E^+_{\kappa_\alpha}(\mathbb R^n)$ is complete in the induced strong topology. However, by
Cartan \cite{Car}, the whole pre-Hil\-bert space $\mathcal
E_{\kappa_\alpha}(\mathbb R^n)$ for $\alpha\in(1,n)$ is strongly incomplete (compare with Theorem~\ref{ZU} above, as well as with Theorems~\ref{complete}, \ref{cor-complete} and Remark~\ref{rem-str} below).

From now on we shall write simply $\alpha$ instead of $\kappa_\alpha$ if $\kappa_\alpha$ serves as an index. For example, $c_\alpha(\cdot)=c_{\kappa_\alpha}(\cdot)$ denotes the
$\alpha$-Riesz inner capacity of a set.

\begin{theorem}\label{th-main-comp}Let a generalized condenser\/ $\mathbf A$ be compact and let each of the potentials\/ $\kappa_\alpha(\cdot,\sigma^i)$, $i\in I$, be continuous on\/ $A_i$, $\boldsymbol{\sigma}\in\mathfrak C(\mathbf A)$ being given. Then in either Case\/~{\rm I} or Case\/~{\rm II} Problem\/~{\rm\ref{pr2}} is solvable for any given\/ $\mathbf a$ and\/ $\mathbf g$, and the class\/ $\mathfrak S^{\boldsymbol{\sigma}}_{\alpha,\mathbf f}(\mathbf A,\mathbf a,\mathbf g)$ of all its solutions is vaguely compact. Furthermore, every minimizing sequence\/ $\{\boldsymbol\mu_k\}_{k\in\mathbb N}\in\mathbb M^{\boldsymbol{\sigma}}_{\alpha,\mathbf f}(\mathbf A,\mathbf a,\mathbf g)$ converges to every\/ $\boldsymbol\lambda\in\mathfrak S^{\boldsymbol{\sigma}}_{\alpha,\mathbf f}(\mathbf A,\mathbf a,\mathbf g)$ strongly in\/ $\mathcal E_\alpha^+(\mathbf A)$; and hence also vaguely provided that all the\/ $A_i$, $i\in I$, are mutually essentially disjoint.\end{theorem}

Theorem~\ref{th-main-comp} is inspired partly by \cite{BC} and will be proved in Section~\ref{th-main-comp-proof}. The following lemma goes back to \cite{Ra,DS}.

\begin{lemma}\label{cont-pot}Let\/ $\mathbf A$ be an arbitrary\/ {\rm(}not necessarily compact\/{\rm)} generalized condenser, and let each of the\/ $\kappa_\alpha(\cdot,\sigma^i)$, $i\in I$, be continuous on\/ $A_i$, $\boldsymbol{\sigma}\in\mathfrak C(\mathbf A)$ being given. Then for every\/ $\boldsymbol\mu\in\mathfrak M^{\boldsymbol{\sigma}}(\mathbf A)$ and every\/ $i\in I$, $\kappa_\alpha(\cdot,\mu^i)$ is continuous on\/~$\mathbb R^n$.\end{lemma}

\begin{proof} Actually, $\kappa_\alpha(\cdot,\sigma^i)$ is continuous on all of $\mathbb R^n$ by \cite[Theorem~1.7]{L}. Since $\kappa_\alpha(\cdot,\mu^i)$ is l.s.c.\ and since $\kappa_\alpha(\cdot,\mu^i)=\kappa_\alpha(\cdot,\sigma^i)-\kappa_\alpha(\cdot,\sigma^i-\mu^i)$ with $\kappa_\alpha(\cdot,\sigma^i)$ continuous and
$\kappa_\alpha(\cdot,\sigma^i-\mu^i)$ l.s.c., $\kappa_\alpha(\cdot,\mu^i)$ is also upper semicontinuous, hence continuous.\end{proof}

\begin{example}\label{ex-2} Let $\mathbf A=(A_i)_{i\in I}$ be as in Example~\ref{ex-1} (see Figure~\ref{fig3.3}), and let $\alpha\in(0,2)$. Also assume that $\mathbf g=\mathbf 1$ and that either Case~II holds or $f_i(x)<\infty$ n.e.\ on~$A_i$, $i=1,2$. Let $\lambda_i$ denote the (unique) $\kappa_\alpha$-capacitary measure on $A_i$ (see Remark~\ref{remark}); then $\kappa_\alpha(\cdot,\lambda_i)$ is continuous on $\mathbb R^n$ and $S^{\lambda_i}_{\mathbb R^n}=A_i$ \cite[Chapter~II, Section~3, n$^\circ$\,13]{L}. For any $\mathbf a=(a_i)_{i\in I}$ define $\sigma^i:=c_i\lambda_i$, $i\in I$, where $a_i<c_i<\infty$. As $\boldsymbol\sigma=(\sigma^i)_{i\in I}$ clearly has finite $\alpha$-Riesz energy, relation (\ref{gauss-finite}) holds by Lemma~\ref{suff-fin}, and since $\mathbf A$ is compact, Problem~\ref{pr2} admits a solution according to Theorem~\ref{th-main-comp}. Thus, no short-cir\-cuit occurs between the oppositely charged plates of the condenser $\mathbf A$, though they intersect each other over the set $\delta_{\mathbf A}=\{\xi_5,\xi_6\}$.\end{example}

\subsection{Proof of Theorem~\ref{th-main-comp}}\label{th-main-comp-proof}

Fix any $\{\boldsymbol\mu_k\}_{k\in\mathbb N}\in\mathbb M^{\boldsymbol{\sigma}}_{\alpha,\mathbf f}(\mathbf A,\mathbf a,\mathbf g)$; it exists because of the (standing) assumption (\ref{gauss-finite}). By Lemma~\ref{lemma-rel-comp}, any of its vague cluster points $\boldsymbol\lambda$ (which exist) belongs to $\mathfrak M^{\boldsymbol{\sigma}}(\mathbf A,\mathbf a,\mathbf g)$. As $\mathfrak M^{\boldsymbol{\sigma}}(\mathbf A,\mathbf a,\mathbf g)$ is in fact sequentially vaguely compact (see Remark~\ref{rem-net}), one can select  a subsequence $\{\boldsymbol\mu_{k_m}\}_{m\in\mathbb N}$ of $\{\boldsymbol\mu_k\}_{k\in\mathbb N}$ such that
\begin{equation}\label{123}\boldsymbol\mu_{k_m}\to\boldsymbol\lambda\text{ \  vaguely as \ }m\to\infty.\end{equation}
Since by Lemma~\ref{cont-pot} each of $\kappa_\alpha(\cdot,\mu_k^i)$ and $\kappa_\alpha(\cdot,\lambda^i)$, $k\in\mathbb N$, $i\in I$, are continuous and hence bounded on the (compact) set $A_j$, $j\in I$, the preceding display yields
\begin{align*}\lim_{m\to\infty}\,\lim_{\ell\to\infty}\kappa_\alpha(\mu_{k_m}^i,\mu_{k_\ell}^j)&=
\lim_{m\to\infty}\,\lim_{\ell\to\infty}\int\kappa_\alpha(\cdot,\mu_{k_m}^i)\,d\mu_{k_\ell}^j=
\lim_{m\to\infty}\,\int\kappa_\alpha(\cdot,\mu_{k_m}^i)\,d\lambda^j\\{}&=\lim_{m\to\infty}\,\int\kappa_\alpha(\cdot,\lambda^j)\,d\mu_{k_m}^i=\kappa_\alpha(\lambda^j,\lambda^i)<\infty\text{ \ for all \ }i,j\in I.\end{align*}
Hence, $\boldsymbol\lambda\in\mathcal E^+_\alpha(\mathbf A)$ and moreover
\begin{equation}\label{123'}\boldsymbol\mu_{k_m}\to\boldsymbol\lambda\text{ \  strongly as \ }m\to\infty.\end{equation}
We assert that this $\boldsymbol\lambda$ solves Problem~\ref{pr2}.

Applying Lemma~\ref{gauss:lsc}, from (\ref{min-seq:}), (\ref{123}) and (\ref{123'}) we obtain
\[-\infty<G_{\alpha,\mathbf f}(\boldsymbol\lambda)\leqslant\lim_{m\to\infty}\,G_{\alpha,\mathbf f}(\boldsymbol\mu_{k_m})=
G^{\boldsymbol\sigma}_{\alpha,\mathbf f}(\mathbf A,\mathbf a,\mathbf g)<\infty,\]
where the first inequality is valid by (\ref{GII}) and (\ref{GI}), while the last one holds according to the (standing) assumption (\ref{gauss-finite}). Hence $\boldsymbol\lambda\in\mathcal E^{\boldsymbol{\sigma}}_{\alpha,\mathbf f}(\mathbf A,\mathbf a,\mathbf g)$, and $\boldsymbol\lambda\in\mathfrak S^{\boldsymbol{\sigma}}_{\alpha,\mathbf f}(\mathbf A,\mathbf a,\mathbf g)$ follows.

Note that the minimizing sequence $\{\boldsymbol\mu_k\}_{k\in\mathbb N}$ is strong Cauchy by Corollary~\ref{cor:fund}. Since a strong Cauchy sequence converges strongly to any of its strong cluster points, we infer from (\ref{123'}) that $\{\boldsymbol\mu_k\}_{k\in\mathbb N}$ converges to $\boldsymbol\lambda$ strongly. The same holds for any other
$\boldsymbol\nu\in\mathfrak S^{\boldsymbol{\sigma}}_{\alpha,\mathbf f}(\mathbf A,\mathbf a,\mathbf g)$, for $\|\boldsymbol\lambda-\boldsymbol\nu\|_{\mathcal E^+_\alpha(\mathbf A)}=0$ according to Lemma~\ref{lemma:unique:}.

To prove  that $\mathfrak S^{\boldsymbol{\sigma}}_{\alpha,\mathbf f}(\mathbf A,\mathbf a,\mathbf g)$ is vaguely compact, consider a sequence $\{\boldsymbol\lambda_k\}_{k\in\mathbb N}$ of its elements. Since it belongs to $\mathbb M^{\boldsymbol{\sigma}}_{\alpha,\mathbf f}(\mathbf A,\mathbf a,\mathbf g)$, it is clear from what has been shown above that any of the vague cluster points
 of $\{\boldsymbol\lambda_k\}_{k\in\mathbb N}$ belongs to $\mathfrak S^{\boldsymbol{\sigma}}_{\alpha,\mathbf f}(\mathbf A,\mathbf a,\mathbf g)$.

Assume finally that the $A_i$, $i\in I$, are mutually essentially disjoint. Then, by Lemma~\ref{lemma:unique:}, a solution to Problem~\ref{pr2} is unique, which yields that the vague cluster set of the given minimizing sequence $\{\boldsymbol\mu_k\}_{k\in\mathbb N}$ reduces to the unique measure $\boldsymbol\lambda$. Since the
vague topology is Hausdorff, $\boldsymbol{\lambda}$ is
actually the vague limit of $\{\boldsymbol{\mu}_k\}_{k\in\mathbb N}$ \cite[Chapter~I, Section~9, n$^\circ$\,1]{B1}.\qed

\section{On the solvability of Problem~\ref{pr2} for Riesz kernels.~II}\label{sec7}

Let $\overline{Q}$ be the closure of $Q\subset\mathbb R^n$ in $\overline{\mathbb R^n}:=\mathbb R^n\cup\{\omega_{\mathbb R^n}\}$, the one-point compactification of $\mathbb R^n$.
The following theorem provides sufficient conditions for the solvability of Problem~\ref{pr2} in the case where $A_i$, $i\in I$, are not necessarily compact; compare with Theorem~\ref{th-main-comp} above.\footnote{In a particular case of a condenser with two oppositely charged plates some of the results presented in this section have been obtained earlier in~\cite{ZBull}.} Regarding the uniqueness of a solution to Problem~\ref{pr2}, see Lemma~\ref{lemma:unique:}.

\begin{theorem}\label{th-main}
Let the set\/ $\overline{A^+}\cap\overline{A^-}$ consist of at most one point, i.e.
\begin{equation}\label{one}\text{either \ }\overline{A^+}\cap\overline{A^-}=\varnothing,\text{ \ or \ } \overline{A^+}\cap\overline{A^-}=\{x_0\}\quad\text{where \ }x_0\in\overline{\mathbb R^n},\end{equation}
and let\/ the given\/ $\mathbf g$ and\/ $\boldsymbol\sigma\in\mathfrak C(\mathbf A)$ satisfy\/ {\rm(\ref{boundd})}. Then in either Case\/~{\rm I} or Case\/~{\rm II} Problem\/~{\rm\ref{pr2}} is solvable for any given\/ $\mathbf a$, and the class of all its solutions is vaguely compact. Furthermore, every minimizing sequence converges to every\/ $\boldsymbol\lambda\in\mathfrak S^{\boldsymbol{\sigma}}_{\alpha,\mathbf f}(\mathbf A,\mathbf a,\mathbf g)$ strongly in\/ $\mathcal E_\alpha^+(\mathbf A)$; hence also vaguely whenever all the\/ $A_i$, $i\in I$, are mutually essentially disjoint.
\end{theorem}

Theorem~\ref{th-main} is sharp in the sense that it is no longer valid if assumption (\ref{boundd}) is omitted (see Theorem~\ref{th:sharp} below).

The proof of Theorem~\ref{th-main} is given in Section~\ref{sec-proof-main}; it is based on the approach that has been developed in Sections~\ref{sec-cond} and~\ref{sec-Gauss-c}, as well as on Theorems~\ref{complete} and~\ref{cor-complete} below providing strong completeness results for semimetric subspaces of $\mathcal E^+_\alpha(\mathbf A)$, properly chosen. In turn, the proofs of Theorems~\ref{complete} and~\ref{cor-complete} substantially use Theorem~\ref{ZU} on the strong completeness of $\mathcal E^+_\alpha(\mathbf A)$ in the case of a standard condenser.

\begin{example}\label{ex-2'} Let $\mathbf A=(A_i)_{i\in I}$ be as in Example~\ref{ex-1'} (see Figure~\ref{fig3.3'}) and let $\alpha=2$; then $c_2(A_i)<\infty$, $i=1,2$  \cite[Example~8.2]{ZPot2}, and hence there exists a (unique) $\kappa_2$-cap\-acitary measure $\lambda_i$ on $A_i$ (see Remark~\ref{remark}). Let $\mathbf g=\mathbf 1$, and let
either Case~II hold or $f_i(x)<\infty$ n.e.\ on $A_i$, $i=1,2$.  For any $\mathbf a=(a_1,a_2)$ define $\sigma^i:=c_i\lambda_i$, $i=1,2$, where $a_i<c_i<\infty$.\footnote{Under the assumptions of Example~\ref{ex-2'}, (\ref{g-mass}) holds since for the given $A_i$, $i=1,2$, and $\kappa=\kappa_2$ we have $S^{\lambda_i}_{\mathbb R^3}=A_i$. This is  seen from the construction of the $\kappa_2$-cap\-acitary measure described in \cite[Theorem~5.1]{L}.} Then $\langle g_i,\sigma^i\rangle=\sigma^i(A_i)=c_i<\infty$, and hence (\ref{boundd}) is fulfilled. As $\boldsymbol\sigma=(\sigma^i)_{i\in I}$ has finite Newtonian energy, (\ref{gauss-finite}) holds by Lemma~\ref{suff-fin}, and Problem~\ref{pr2} admits a solution according to Theorem~\ref{th-main}, which is unique by Lemma~\ref{lemma:unique:}. Thus, no short-cir\-cuit occurs between $A_1$ and $A_2$, though these oppositely charged conductors touch each other at the point $\omega_{\mathbb R^3}$. Note that, although $c_2(A)<\infty$, Theorem~6.2 from \cite{Z9} on the solvability of Problem~\ref{pr2} for a standard condenser cannot be applied, which is caused by the unboundedness of the Coulomb kernel $\kappa_2$ on~$A_1\times A_2$.\end{example}

\subsection{Strong completeness theorems for semimetric subspaces of $\mathcal E^+_\alpha(\mathbf A)$}\label{sec-str}
The (convex) sets
$\mathcal E^+_{\alpha}(\mathbf A,\leqslant\!\mathbf a,\mathbf g):=\mathcal E^+_{\alpha}(\mathbf A)\cap\mathfrak M^+(\mathbf A,\leqslant\!\mathbf a,\mathbf g)$ and
$\mathcal E^{\boldsymbol\sigma}_{\alpha}(\mathbf A,\mathbf a,\mathbf g):=\mathcal E^+_{\alpha}(\mathbf A)\cap\mathfrak M^{\boldsymbol\sigma}(\mathbf A,\mathbf a,\mathbf g)$,
$\boldsymbol\sigma\in\mathfrak C(\mathbf A)$ being given,
can certainly be thought of as semimetric subspaces of $\mathcal E^+_{\alpha}(\mathbf A)$; their topologies will likewise be called {\it strong}.

\begin{theorem}\label{complete} Suppose that a generalized condenser\/ $\mathbf A$ satisfies condition\/ {\rm(\ref{one})}. Then for any given\/ $\mathbf g$ and \/ $\mathbf a$ the semimetric space\/ $\mathcal E^+_\alpha(\mathbf A,\leqslant\!\mathbf a,\mathbf g)$ is strongly complete. In more detail, any strong Cauchy sequence\/ $\{\boldsymbol\mu_k\}_{k\in\mathbb N}\subset\mathcal E^+_\alpha(\mathbf A,\leqslant\!\mathbf a,\mathbf g)$ converges strongly to any of its vague cluster points. If moreover all the\/ $A_i$, $i\in I$, are mutually essentially disjoint, then the strong topology on the space\/ $\mathcal E^+_\alpha(\mathbf A,\leqslant\!\mathbf a,\mathbf g)$ is finer than the induced vague topology.\end{theorem}

We first outline the scheme of the proof of Theorem~\ref{complete}. In view of Lemma~\ref{lemma-rel-comp} on the vague compactness of $\mathfrak M^+(\mathbf A,\leqslant\!\mathbf a,\mathbf g)$, we can assume that a strong Cauchy sequence $\{\boldsymbol{\mu}_k\}_{k\in\mathbb N}\subset\mathcal E^+_\alpha(\mathbf A,\leqslant\!\mathbf a,\mathbf g)$ converges vaguely to $\boldsymbol\mu\in\mathfrak M^+(\mathbf A,\leqslant\!\mathbf a,\mathbf g)$. It remains to show that $\boldsymbol{\mu}_k\to\boldsymbol\mu$ in the strong topology of $\mathcal E^+_\alpha(\mathbf A)$, which by the isometry between $\mathcal E^+_\alpha(\mathbf A)$ and its $R$-image is equivalent to the assertion $R\boldsymbol{\mu}_k\to R\boldsymbol\mu$ strongly in $\mathcal E_\alpha(\mathbb R^n)$. The difficulty appearing here is the strong {\it incompleteness\/} of $\mathcal E_\alpha(\mathbb R^n)$. However, if $\overline{A^+}\cap\overline{A^-}$ consists of at most $\omega_{\mathbb R^n}$, the completeness of the metric space of all $\nu\in\mathcal E_\alpha(\mathbb R^n)$ such that $\nu^\pm$ are supported by $A^\pm$ was shown in \cite[Theorem~1]{ZUmzh}. The remaining case $\overline{A^+}\cap\overline{A^-}=\{x_0\}$, $x_0\ne\omega_{\mathbb R^n}$, is reduced to the case $\overline{A^+}\cap\overline{A^-}=\{\omega_{\mathbb R^n}\}$ with the aid of the Kelvin transformation relative to $S(x_0,1)$.

\begin{proof} Fix a strong Cauchy sequence $\{\boldsymbol{\mu}_k\}_{k\in\mathbb N}\subset\mathcal E^+_\alpha(\mathbf A,\leqslant\!\mathbf a,\mathbf g)$. By Lemma~\ref{lemma-rel-comp}, for any of its vague cluster points $\boldsymbol\mu$ (which exist) we have $\boldsymbol{\mu}\in\mathfrak M^+(\mathbf A,\leqslant\!\mathbf a,\mathbf g)$. As $\mathfrak M^+(\mathbf A,\leqslant\!\mathbf a,\mathbf g)$ is sequentially vaguely closed (see Remark~\ref{rem-net}), one can choose a  subsequence
$\{\boldsymbol{\mu}_{k_m}\}_{m\in\mathbb N}$ of $\{\boldsymbol{\mu}_k\}_{k\in\mathbb N}$ converging vaguely to the measure $\boldsymbol{\mu}$, i.e.
\begin{equation}\label{conv-vag}\mu_{k_m}^i\to\mu^i\text{ \ vaguely in \ }\mathfrak M(X), \ i\in I.\end{equation}
It is obvious that $\{\boldsymbol{\mu}_{k_m}\}_{m\in\mathbb N}$ is likewise strong Cauchy in $\mathcal E^+_{\alpha}(\mathbf A)$.

We proceed by showing that $\kappa_\alpha(\boldsymbol{\mu},\boldsymbol{\mu})$ is finite, hence
\begin{equation}\label{belongs}\boldsymbol{\mu}\in\mathcal E^+_\alpha(\mathbf A,\leqslant\!\mathbf a,\mathbf g),\end{equation}
and moreover that $\boldsymbol{\mu}_{k_m}\to\boldsymbol{\mu}$ strongly as $m\to\infty$, i.e.
\begin{equation}\label{conv-str}\lim_{m\to\infty}\,\|\boldsymbol{\mu}_{k_m}-\boldsymbol{\mu}\|_{\mathcal E^+_{\alpha}(\mathbf A)}=0.\end{equation}

Assume first that $\overline{A^+}\cap\overline{A^-}$ either is empty or coincides with $\{\omega_{\mathbb R^n}\}$. Then $\mathbf A$ forms a standard condenser in $\mathbb R^n$ and hence, by (\ref{conv-vag}),  $R\boldsymbol{\mu}_{k_m}\to R\boldsymbol{\mu}$ (as $m\to\infty$) in the vague topology of $\mathfrak M(\mathbb R^n)$. Noting that $\{R\boldsymbol{\mu}_{k_m}\}_{m\in\mathbb N}$ is a strong Cauchy sequence in $\mathcal E_\alpha(\mathbb R^n)$, we conclude from \cite[Theorem~1 and Corollary~1]{ZUmzh} (see also Theorem~\ref{ZU} above) that there exists a unique $\eta\in\mathcal E_\alpha(\mathbb R^n)$ such that
\[R\boldsymbol{\mu}_{k_m}\to\eta\quad\text{strongly and vaguely as \ }m\to\infty.\]
As the vague topology on $\mathfrak M(\mathbb R^n)$ is Hausdorff, we thus have $\eta=R\boldsymbol{\mu}$, which in view of (\ref{Re}), (\ref{isom}) and the last display results in (\ref{belongs}) and (\ref{conv-str}).

We next proceed by analyzing the case
\begin{equation}\label{ex}\overline{A^+}\cap\overline{A^-}=\{x_0\}\quad\text{where \ }x_0\in\mathbb R^n.\end{equation}
Consider the inversion $I_{x_0}$ with respect to $S(x_0,1)$; namely, each point $x\ne
x_0$ is mapped to the point $x^*$ on the ray through $x$ which issues from $x_0$, determined uniquely by
\[|x-x_0|\cdot|x^*-x_0|=1.\]
This is a self-homeomorphism of $\mathbb
R^n\setminus\{x_0\}$; furthermore,
\begin{equation}\label{inv}|x^*-y^*|=\frac{|x-y|}{|x-x_0||y-x_0|}.\end{equation}
Extend it to a self-homeomorphism of $\overline{\mathbb R^n}$ by setting $I_{x_0}(x_0)=\omega_{\mathbb R^n}$ and $I_{x_0}(\omega_{\mathbb R^n})=x_0$.

To each (signed) scalar measure $\nu\in\mathfrak M(\mathbb R^n)$ with
$\nu(\{x_0\})=0$, in particular for every $\nu\in\mathcal E_\alpha(\mathbb R^n)$, there corresponds the Kelvin transform
$\nu^*\in\mathfrak M(\mathbb R^n)$ by means of the
formula
\[d\nu^*(x^*)=|x-x_0|^{\alpha-n}\,d\nu(x),\quad x^*\in\mathbb R^n\]
(see~\cite{R} or \cite[Chapter IV, Section 5, n$^\circ$\,19]{L}). Then, in consequence of (\ref{inv}),
\begin{equation*}\label{KP}\kappa_\alpha(x^*,\nu^*)=|x-x_0|^{n-\alpha}\kappa_\alpha(x,\nu),\quad x^*\in\mathbb R^n,\end{equation*}
and therefore
\begin{equation}\label{K}\kappa_\alpha(\nu^*,\nu_1^*)=\kappa_\alpha(\nu,\nu_1)\end{equation}
for every $\nu_1\in\mathfrak M(\mathbb R^n)$ that does not have an atomic mass at $x_0$. It is clear that the Kelvin transformation is additive and that it is an involution, i.e.
\begin{align}\label{A}\bigl(\nu+\nu_1\bigr)^*&=
\nu^*+\nu_1^*,\\
\label{AA}(\nu^*)^*&=\nu.\end{align}

Write $A_i^*:=I_{x_0}\bigl(\,\overline{A_i}\,\bigr)\cap\mathbb R^n$ and ${\rm sign}\,A_i^*:={\rm sign}\,A_i=s_i$ for each $i\in I$; then $\mathbf A^*=(A_i^*)_{i\in I}$ forms a standard condenser in $\mathbb R^n$, which is seen from (\ref{ex}) in view of the properties of~$I_{x_0}$.

Applying the Kelvin transformation to each of the components $\nu^i$ of any given $\boldsymbol{\nu}=(\nu^i)_{i\in I}\in\mathcal E_\alpha^+(\mathbf A)$ we get $\boldsymbol{\nu}^*:=(\nu^i)^*_{i\in I}\in\mathfrak M^+(\mathbf A^*)$.
Based on Lemma~\ref{l-Ren}, identity (\ref{isom}) and relations (\ref{K})--(\ref{AA}), we also see that
the $\alpha$-Riesz energy of $\boldsymbol{\nu}^*$ is finite, and furthermore
\begin{equation}\label{pres}\|\boldsymbol{\nu}_1^*-\boldsymbol{\nu}_2^*\|_{\mathcal E^+_\alpha(\mathbf A^*)}=\|\boldsymbol{\nu}_1-\boldsymbol{\nu}_2\|_{\mathcal E^+_\alpha(\mathbf A)}\quad\text{for all \ }\boldsymbol{\nu}_1,\boldsymbol{\nu}_2\in\mathcal E^+_\alpha(\mathbf A).\end{equation}
Summarizing the above, because of (\ref{AA}) we arrive at the following observation: {\it the Kelvin transformation is a bijective isometry of\/ $\mathcal E^+_\alpha(\mathbf A)$ onto\/ $\mathcal E^+_\alpha(\mathbf A^*)$}.

Let $\boldsymbol{\mu}_{k_m}$, $m\in\mathbb N$, and $\boldsymbol{\mu}$ be the measures chosen at the beginning of the proof. In view of (\ref{bbbound}) and (\ref{conv-vag}), for each $i\in I$ one can apply \cite[Lemma~4.3]{L} to
$\mu_{k_m}^i$, $m\in\mathbb N$, and $\mu^i$, and consequently
\begin{equation}\label{vague'}\boldsymbol{\mu}_{k_m}^*\to\boldsymbol{\mu}^*\quad\text{vaguely as \
}m\to\infty.\end{equation}
But $\{\boldsymbol{\mu}_{k_m}^*\}_{m\in\mathbb N}$ is a strong Cauchy sequence in $\mathcal E_\alpha^+(\mathbf A^*)$, which is clear from (\ref{pres}). This together with (\ref{vague'}) implies with the aid of Theorem~\ref{ZU} that $\boldsymbol{\mu}^*\in\mathcal E_\alpha^+(\mathbf A^*)$ and also that
\[\lim_{m\to\infty}\,\|\boldsymbol{\mu}_{k_m}^*-\boldsymbol{\mu}^*\|_{\mathcal E_\alpha^+(\mathbf A^*)}=0.\]
Another application of the above observation then leads to (\ref{belongs}) and (\ref{conv-str}), as was to be proved.

In turn, (\ref{conv-str}) implies that $\boldsymbol{\mu}_k\to\boldsymbol{\mu}$ strongly in $\mathcal E^+_\alpha(\mathbf A)$ as $k\to\infty$, for $\{\boldsymbol{\mu}_k\}_{k\in\mathbb N}$ is strong Cauchy and hence converges strongly to any of its strong cluster points.
It has thus been established that $\{\boldsymbol{\mu}_k\}_{k\in\mathbb N}$ converges strongly to any of its vague cluster points, which is the first assertion of the theorem. Assume now that all the $A_i$, $i\in I$, are mutually essentially disjoint. Then $\|\boldsymbol{\mu}_1-\boldsymbol{\mu}_2\|_{\mathcal E^+_{\alpha}(\mathbf A)}$ is a metric (Theorem~\ref{lemma:semimetric}), and hence $\boldsymbol{\mu}$ has to be the unique vague cluster point of $\{\boldsymbol{\mu}_k\}_{k\in\mathbb N}$. Since the
vague topology is Hausdorff, $\boldsymbol{\mu}$ is
actually the vague limit of $\{\boldsymbol{\mu}_k\}_{k\in\mathbb N}$ \cite[Chapter~I, Section~9, n$^\circ$\,1]{B1}.\end{proof}

\begin{theorem}\label{cor-complete}Given\/ $\mathbf A$, $\mathbf g$, and\/ $\boldsymbol\sigma\in\mathfrak C(\mathbf A)$, assume that\/ {\rm(\ref{boundd})} and\/ {\rm(\ref{one})} both hold. Then for every vector\/ $\mathbf a$ the semimetric space\/ $\mathcal E^{\boldsymbol\sigma}_\alpha(\mathbf A,\mathbf a,\mathbf g)$ is strongly complete.\end{theorem}

\begin{proof}Fix a strong Cauchy sequence $\{\boldsymbol{\mu}_k\}_{k\in\mathbb N}\subset\mathcal E^{\boldsymbol\sigma}_\alpha(\mathbf A,\mathbf a,\mathbf g)$. By Lemma~\ref{lemma-rel-cl}, any of its vague cluster points $\boldsymbol{\mu}$ (which exist) belongs to $\mathfrak M^{\boldsymbol\sigma}(\mathbf A,\mathbf a,\mathbf g)$, while according to Theorem~\ref{complete} it has finite $\alpha$-Riesz energy and moreover $\boldsymbol{\mu}_k\to\boldsymbol{\mu}$ strongly as $k\to\infty$.\end{proof}

\begin{remark}Theorem~\ref{cor-complete} does not remain valid if assumption (\ref{boundd}) is omitted from its hypotheses. This is seen from the proof of Theorem~\ref{th:sharp} (see Section~\ref{sec:sharp} below).\end{remark}

\begin{remark}\label{rem-str}Since either of the semimetric spaces $\mathcal
E^+_\alpha(\mathbf A,\leqslant\!\mathbf a,\mathbf g)$ and $\mathcal E^{\boldsymbol\sigma}_\alpha(\mathbf A,\mathbf a,\mathbf g)$ is isometric to its
$R$-image, Theorems~\ref{complete} and~\ref{cor-complete} have singled out strongly
complete topological subspaces of the pre-Hilbert space $\mathcal
E_\alpha(\mathbb R^n)$, whose elements are (signed) Radon measures. This is of independent interest because, by Cartan, $\mathcal E_\alpha(\mathbb R^n)$ is strongly incomplete.\end{remark}

\subsection{Proof of Theorem~\ref{th-main}}\label{sec-proof-main}
Fix any $\{\boldsymbol\mu_k\}_{k\in\mathbb N}\in\mathbb M^{\boldsymbol{\sigma}}_{\alpha,\mathbf f}(\mathbf A,\mathbf a,\mathbf g)$; it exists because of assumption (\ref{gauss-finite}), and it is strong Cauchy in the semimetric space $\mathcal
E^{\boldsymbol{\sigma}}_\alpha(\mathbf A,\mathbf a,\mathbf g)$ by Corollary~\ref{cor:fund}. Since $\mathfrak M^+(\mathbf A)$ is sequentially vaguely closed (see Remark~\ref{rem-net}), by Lemma~\ref{lemma-rel-cl} there is a subsequence
$\{\boldsymbol{\mu}_{k_m}\}_{m\in\mathbb N}$ of $\{\boldsymbol{\mu}_k\}_{k\in\mathbb N}$ converging vaguely to some $\boldsymbol\mu\in\mathfrak M^{\boldsymbol{\sigma}}(\mathbf A,\mathbf a,\mathbf g)$, while by Theorem~\ref{cor-complete} we actually have $\boldsymbol\mu\in\mathcal
E^{\boldsymbol{\sigma}}_\alpha(\mathbf A,\mathbf a,\mathbf g)$ and
\begin{equation}\label{starr1}\lim_{k\to\infty}\,\|\boldsymbol{\mu}_k-\boldsymbol{\mu}\|_{\mathcal E^+_\alpha(\mathbf A)}=0.\end{equation}
Also note that, by relations (\ref{gauss-finite}), (\ref{min-seq:}) and Lemma~\ref{gauss:lsc},
\[-\infty<G_{\alpha,\mathbf f}(\boldsymbol{\mu})\leqslant\lim_{m\to\infty}\,G_{\alpha,\mathbf f}(\boldsymbol{\mu}_{k_m})=G^{\boldsymbol{\sigma}}_{\alpha,\mathbf f}(\mathbf A,\mathbf a,\mathbf g)<\infty,\]
the first inequality being valid according to (\ref{GII}) and (\ref{GI}). Thus $\boldsymbol{\mu}\in\mathcal E^{\boldsymbol{\sigma}}_{\alpha,\mathbf f}(\mathbf A,\mathbf a,\mathbf g)$ and therefore $G_{\alpha,\mathbf f}(\boldsymbol{\mu})\geqslant
G^{\boldsymbol{\sigma}}_{\alpha,\mathbf f}(\mathbf A,\mathbf a,\mathbf g)$. All this combined implies that $\boldsymbol{\mu}\in\mathfrak S^{\boldsymbol{\sigma}}_{\alpha,\mathbf f}(\mathbf A,\mathbf a,\mathbf g)$.

To verify that $\mathfrak S^{\boldsymbol{\sigma}}_{\alpha,\mathbf f}(\mathbf A,\mathbf a,\mathbf g)$ is vaguely compact, fix a sequence
$\{\boldsymbol\lambda_k\}_{k\in\mathbb N}$ of its elements. By Lemma~\ref{lemma:unique:}, it is strong Cauchy in $\mathcal E^{\boldsymbol{\sigma}}_\alpha(\mathbf A,\mathbf a,\mathbf g)$, and the same arguments as above show that the (nonempty) vague cluster set of $\{\boldsymbol\lambda_k\}_{k\in\mathbb N}$ is contained in $\mathfrak S^{\boldsymbol{\sigma}}_{\alpha,\mathbf f}(\mathbf A,\mathbf a,\mathbf g)$.

If $\boldsymbol\lambda$ is any element of $\mathfrak S^{\boldsymbol{\sigma}}_{\alpha,\mathbf f}(\mathbf A,\mathbf a,\mathbf g)$, then by Lemma~\ref{lemma:unique:}, $\boldsymbol\lambda$ belongs to the $R$-equivalence class $[\boldsymbol{\mu}]$, which in view of (\ref{starr1}) implies that $\boldsymbol\mu_k\to\boldsymbol\lambda$ strongly in $\mathcal E_\alpha^+(\mathbf A)$. Assuming now that all the $A_i$, $i\in I$, are mutually essentially disjoint, we see from the last assertion of Theorem~\ref{complete} that $\boldsymbol\mu_k\to\boldsymbol\lambda$ also vaguely.\qed

\subsection{On the sharpness of Theorem~\ref{th-main}}\label{sec:sharp}

The purpose of this section is to show that Theorem~\ref{th-main} on the solvability of Problem~\ref{pr2} is no longer valid if assumption (\ref{boundd}) is dropped from its hypotheses.

\begin{definition}[{\rm see  \cite[Theorem~VII.13]{Brelo2}}] A closed set $F\subset\mathbb R^n$ is said to be {\it $\alpha$-thin
at\/} $\omega_{\mathbb R^n}$ if either $F$ is compact, or the inverse of $F$ relative to $S(0,1)$ has $x=0$ as an $\alpha$-irregular boundary point (cf.\ \cite[Theorem~5.10]{L}).\end{definition}

\begin{remark}\label{rem-thin} Alternatively, by Wiener's criterion of $\alpha$-irregularity of a point, a closed set $F\subset\mathbb R^n$ is $\alpha$-thin at $\omega_{\mathbb R^n}$ if and only if
\[\sum_{k\in\mathbb N}\,\frac{c_\alpha(F_k)}{q^{k(n-\alpha)}}<\infty,\]
where $q>1$ and $F_k:=F\cap\{x\in\mathbb R^n: q^k\leqslant|x|<q^{k+1}\}$. Since, by \cite[Lemma~5.5]{L}, $c_\alpha(F)=\infty$ is equivalent to the relation
\[\sum_{k\in\mathbb N}\,c_\alpha(F_k)=\infty,\]
one can define a closed set $Q\subset\mathbb R^n$ with $c_\alpha(Q)=\infty$, but $\alpha$-thin at $\omega_{\mathbb R^n}$ (see also \cite[pp.~276--277]{Ca2}).
In the case $n=3$ and $\alpha=2$, such a $Q$ can be given as follows:
\[Q:=\bigl\{x\in\mathbb R^3: \ 0\leqslant x_1<\infty, \ x_2^2+x_3^2\leqslant\exp(-2x_1^r)\text{ \ with \ }r\in(0,1]\bigr\};\]
note that $Q$ thus defined has finite $c_\alpha(\cdot)$ if $r$ in its definition is ${}>1$ \cite[Example~8.2]{ZPot2}.\end{remark}

Assume for simplicity that $\mathbf g=\mathbf1$, $\mathbf a=\mathbf1$ and $\mathbf f=\mathbf 0$.
Furthermore, let $0<\alpha\leqslant2$, $I^+=\{1\}$, $I^-=\{2\}$, and let $A_2\subset\mathbb R^n$ be a closed set with $c_\alpha(A_2)=\infty$, though $\alpha$-thin at $\omega_{\mathbb R^n}$ (see Remark~\ref{rem-thin}). Assume moreover that the (open) set $D:=A_2^c$ is connected and that $A_1$ is a compact subset of $D$ with $c_\alpha(A_1)>0$. Given the (standard) condenser $\mathbf A:=(A_1,A_2)$ and a constraint $\boldsymbol\sigma\in\mathfrak C(\mathbf A)$, let $\mathcal E^{\boldsymbol\sigma}_\alpha(\mathbf A,\mathbf 1)$ stand for the class of vector measures admissible in Problem~\ref{pr2} with those data.
The sharpness of condition (\ref{boundd}) for the validity of Theorem~\ref{th-main} is illustrated by the following assertion.

\begin{theorem}\label{th:sharp}Under the above assumptions there exists a constraint\/ $\boldsymbol\sigma\in\mathfrak C(\mathbf A)$ with\/ $1<\sigma^1(A_1)<\infty$ and\/ $\sigma^2(A_2)=\infty$ such that
\[\kappa_\alpha(\boldsymbol\nu,\boldsymbol\nu)>\inf_{\boldsymbol\mu\in\mathcal E^{\boldsymbol\sigma}_\alpha(\mathbf A,\mathbf 1)}\,\kappa_\alpha(\boldsymbol\mu,\boldsymbol\mu)=:w^{\boldsymbol\sigma}_\alpha(\mathbf A,\mathbf 1)\text{ \ for every \ }\boldsymbol\nu\in\mathcal E^{\boldsymbol\sigma}_\alpha(\mathbf A,\mathbf 1).\]
\end{theorem}

Crucial for the proof to be given below is to show that for a certain $\boldsymbol\sigma\in\mathfrak C(\mathbf A)$, $w^{\boldsymbol\sigma}_\alpha(\mathbf A,\mathbf 1)$ equals the $G$-energy of the $G$-capacitary measure $\lambda$ on $A_1$, where $G$ denotes the $\alpha$-Green kernel on $D$. Intuitively this is clear since for any given $\mu\in\mathcal E^+_\alpha(A_1,1)$, $\mu-\mu'$ minimizes the $\alpha$-Riesz energy among all $\mu-\nu$, $\nu$ ranging over $\mathcal E^+_\alpha(A_2)$. (Here and in the sequel $\mu'$ denotes the {\it $\alpha$-Riesz balayage\/} of $\mu\in\mathfrak M^+(\mathbb R^n)$ onto $A_2$, uniquely determined in the frame of the classical approach by \cite[Theorem~3.6]{FZ}.) And since $\|\mu-\mu'\|_\alpha=\|\mu\|_G$, further minimizing over $\mu\in\mathcal E^+_\alpha(A_1,1)$ (equivalently, over $\mathcal E^+_g(A_1,1)$) leads to the claimed equality. However, $q:=1-\lambda'(A_2)>0$, because $A_2$ is $\alpha$-thin at $\omega_{\mathbb R^n}$. Thus, in view of $c_\alpha(A_2)=\infty$, one can choose $\tau_\ell\in\mathcal E^+_\alpha(A_2,q)$, $\ell\in\mathbb N$, so that $\|\tau_\ell\|_\alpha\to0$ as $\ell\to\infty$ and $S^{\tau_\ell}_{\mathbb R^n}\subset B(0,\ell)^c$. With a constraint $\boldsymbol\sigma$ properly chosen, the sequence $(\lambda,\lambda'+\tau_\ell)\in\mathcal E^+_\alpha(\mathbf A;\mathbf 1)$, $\ell\in\mathbb N$, is therefore minimizing, but it converges vaguely and strongly to $(\lambda,\lambda')$ which is not admissible.

\begin{proof}Denote by $G=G^\alpha_D$ the $\alpha$-Green kernel on the locally compact space $D$, defined by
\[G^\alpha_D(x,y):=\kappa_\alpha(x,\varepsilon_y)-\kappa_\alpha(x,\varepsilon_y'), \ x,y\in D,\]
$\varepsilon_y$ being the unit Dirac measure at a point $y$ \cite{L,FZ}. Since $c_\alpha(A_1)$ is ${}>0$, so is $c_G(A_1)$ \cite[Lemma~2.6]{DFHSZ}, and hence, by the compactness of $A_1$, there is $\lambda\in\mathcal E^+_G(A_1,1)$ with
\[G(\lambda,\lambda)=c_G(A_1)^{-1}<\infty.\]
In fact, such a $\lambda$ is unique, for the $\alpha$-Green kernel is strictly positive definite \cite[Theorem~4.9]{FZ}. As $S^\lambda_D$ is compact, it is seen from \cite[Lemmas~3.5, 3.6]{DFHSZ2} that both $\lambda$ and $\lambda'$ have finite $\alpha$-Riesz energy and moreover
\[\|\lambda\|_G=\|\lambda-\lambda'\|_\alpha.\]
Finally, since $D^c=A_2$ is $\alpha$-thin at $\omega_{\mathbb R^n}$, from \cite[Theorem~3.21]{FZ} (see also the earlier papers \cite[Theorem~B]{Z0} and \cite[Theorem~4]{Z1}) we get
\begin{equation}\label{q}q:=\lambda(A_1)-\lambda'(A_2)>0.\end{equation}

Consider an exhaustion of $A_2$ by an increasing sequence of compact sets $K_\ell$, $\ell\in\mathbb N$. Since $c_\alpha(A_2)=\infty$, the strict positive definiteness of the $\alpha$-Riesz kernel and the subadditivity of $c_\alpha(\cdot)$ on universally measurable sets yield $c_\alpha(A_2\setminus
K_\ell)=\infty$ for all $\ell\in\mathbb N$. Hence, for every $\ell$ one can choose a measure
$\tau_\ell\in\mathcal E_\alpha^+(A_2\setminus K_\ell,q)$ with compact support so that
\[\lim_{\ell\to\infty}\,\|\tau_\ell\|_\alpha=0.\]
Certainly, there is no loss of generality in assuming $K_\ell\cup S_{\mathbb R^n}^{\tau_\ell}\subset K_{\ell+1}$.

Choose a constraint
\[\sigma^1:=\lambda+\delta_1,\quad\sigma^2:=\lambda'+\sum_{\ell\in\mathbb N}\,\tau_\ell+\delta_2,\]
where $\delta_i$, $i=1,2$, is a positive bounded Radon measure whose (closed) support coincides with $A_i$.\footnote{Such a $\delta_i$ can be constructed as follows. Consider a sequence of points $x_j$ of $A_i$ which is dense in $A_i$ and define $\delta_i=\sum_{j\in\mathbb N}\,2^{-j}\varepsilon_{x_j}$. Actually, the summands $\delta_i$, $i=1,2$, are added only in order to satisfy (\ref{g-mass}).}
We assert that the problem of minimizing $\kappa_\alpha(\boldsymbol\mu,\boldsymbol\mu)$ over the class $\mathcal E^{\boldsymbol\sigma}_\alpha(\mathbf A,\mathbf 1)$ with the constraint $\boldsymbol\sigma$ thus defined is unsolvable.

It follows from the above that $\{\boldsymbol\mu_\ell\}_{\ell\in\mathbb N}$ with $\mu_\ell^1=\lambda$ and $\mu_\ell^2=\lambda'+\tau_\ell$, $\ell\in\mathbb N$, belongs to $\mathcal E^{\boldsymbol\sigma}_\alpha(\mathbf A,\mathbf 1)$, so that
\begin{equation}\label{limmm}\kappa_\alpha(\boldsymbol\mu_\ell,\boldsymbol\mu_\ell)\geqslant w_\alpha^{\boldsymbol\sigma}(\mathbf A,\mathbf 1)\quad\text{for all \ }\ell\in\mathbb N,\end{equation}
and moreover
\begin{equation}\label{limm}\lim_{\ell\to\infty}\,\kappa_\alpha(\boldsymbol\mu_\ell,\boldsymbol\mu_\ell)=\lim_{\ell\to\infty}\,\|\lambda-\lambda'-\tau_\ell\|^2_\alpha=
\|\lambda-\lambda'\|^2_\alpha=
\|\lambda\|^2_G.\end{equation}
On the other hand, for any $\zeta\in\mathcal E^+_\alpha(\mathbb R^n)$ the balayage $\zeta'$ is in fact the orthogonal projection of $\zeta$ onto the convex cone $\mathcal E^+_\alpha(D^c)$, i.e.
\begin{equation*}\|\zeta-\zeta'\|_\alpha<\|\zeta-\nu\|_\alpha\quad\text{for all \ }\nu\in\mathcal E^+_\alpha(D^c), \ \nu\ne\zeta'\end{equation*}
(see \cite[Theorem~4.12]{Fu5} or \cite[Theorem~3.1]{FZ}). For any $\boldsymbol\mu\in\mathcal E^{\boldsymbol\sigma}_\alpha(\mathbf A,\mathbf 1)$ we therefore obtain
\begin{equation}\label{li}\kappa_\alpha(\boldsymbol\mu,\boldsymbol\mu)=\|\mu^1-\mu^2\|^2_\alpha\geqslant\|\mu^1-(\mu^1)'\|^2_\alpha=
\|\mu^1\|^2_G\geqslant\|\lambda\|^2_G,\end{equation}
which in view of the arbitrary choice of $\boldsymbol\mu\in\mathcal E^{\boldsymbol\sigma}_\alpha(\mathbf A,\mathbf 1)$ yields
$w_\alpha^{\boldsymbol\sigma}(\mathbf A,\mathbf 1)\geqslant\|\lambda\|^2_G$. As the converse inequality holds in consequence of (\ref{limmm}) and (\ref{limm}), we actually have
\[w_\alpha^{\boldsymbol\sigma}(\mathbf A,\mathbf 1)=\|\lambda\|^2_G.\]

To complete the proof, assume on the contrary that the extremal problem under consideration is solvable. Then there exists the (unique) $\boldsymbol\mu\in\mathcal E^{\boldsymbol\sigma}_\alpha(\mathbf A,\mathbf 1)$ such that all the inequalities in (\ref{li}) are in fact equalities. But this is possible only provided that both
$\lambda=\mu^1$ and $\lambda'(D^c)=1=\lambda(A_1)$ hold, which contradicts (\ref{q}).\end{proof}

\section{On continuity of the minimizers $\boldsymbol\lambda^{\boldsymbol{\sigma}}_{\mathbf A}$ with
respect to $(\mathbf{A},\boldsymbol{\sigma})$}\label{sec:cont}

Recall that we are working with the $\alpha$-Riesz kernel $\kappa_\alpha(x,y)=|x-y|^{\alpha-n}$ of order $\alpha\in(0,n)$ on $\mathbb R^n$, $n\geqslant 3$. Given an arbitrary (generalized) condenser $\mathbf A=(A_i)_{i\in I}$, fix a sequence of (generalized) condensers $\mathbf A_\ell:=(A_i^\ell)_{i\in I}$, $\ell\in\mathbb N$, with ${\rm sign}\,A_i^\ell={\rm sign}\,A_i$ such that
\[A_i^{\ell+1}\subset A_i^{\ell}\text{ \ and \ }A_i=\bigcap_{k\in\mathbb N}\,A_i^k\text{ \ for any $i\in I$, \ $\ell\in\mathbb N$}.\]
Fix also constraints $\boldsymbol\sigma=(\sigma^i)_{i\in I}\in\mathfrak C(\mathbf A)$ and $\boldsymbol\sigma_\ell=(\sigma_\ell^i)_{i\in I}\in\mathfrak C(\mathbf A_\ell)$ with the properties that $\sigma_{\ell}^i\geqslant\sigma_{\ell+1}^i\geqslant\sigma^i$ for all $\ell\in\mathbb N$ and $i\in I$, and
\begin{equation}\label{eq:vague}\boldsymbol\sigma_\ell\to\boldsymbol\sigma\text{ \ vaguely as \ }\ell\to\infty.\end{equation}
Then the following statement on continuity holds.

\begin{theorem}\label{th:cont}Assume in addition that for a certain\/ $\ell_0\in\mathbb N$ all the hypotheses of Theorem\/~{\rm\ref{th-main-comp}} or Theorem\/~{\rm\ref{th-main}} hold for\/ $\mathbf A_{\ell_0}$ and\/ $\boldsymbol\sigma_{\ell_0}$ in place of\/ $\mathbf A$ and\/ $\boldsymbol\sigma$.
Then
\begin{equation}\label{sigma}
G^{\boldsymbol{\sigma}}_{\alpha,\mathbf{f}}(\mathbf{A},\mathbf{a},\mathbf{g})=
\lim_{\ell\to\infty}\,G^{\boldsymbol{\sigma}_\ell}_{\alpha,\mathbf{f}}(\mathbf{A}_\ell,\mathbf{a},\mathbf{g}).\end{equation}
Fix\/ $\boldsymbol{\lambda}^{\boldsymbol{\sigma}}_{\mathbf{A}}\in\mathfrak
S^{\boldsymbol{\sigma}}_{\alpha,\mathbf{f}}(\mathbf{A},\mathbf{a},\mathbf{g})$, and for every\/ $\ell\geqslant\ell_0$ fix\/
$\boldsymbol{\lambda}^{\boldsymbol{\sigma}_{\ell}}_{\mathbf{A}_{\ell}}\in\mathfrak
S^{\boldsymbol{\sigma}_{\ell}}_{\alpha,\mathbf{f}}(\mathbf{A}_{\ell},\mathbf{a},\mathbf{g})$;
such solutions to Problem\/~{\rm\ref{pr2}} with the corresponding data exist. Then for every $m\geqslant\ell_0$ the {\rm(}nonempty{\rm)} vague cluster set of the sequence\/
$\bigl\{\boldsymbol{\lambda}^{\boldsymbol{\sigma}_{\ell}}_{\mathbf{A}_{\ell}}\bigr\}_{\ell\geqslant m}$ is contained in\/ $\mathfrak
S^{\boldsymbol{\sigma}}_{\alpha,\mathbf{f}}(\mathbf{A},\mathbf{a},\mathbf{g})$.
Furthermore,
$\boldsymbol{\lambda}^{\boldsymbol{\sigma}_{\ell}}_{\mathbf{A}_{\ell}}\to
\boldsymbol{\lambda}^{\boldsymbol{\sigma}}_{\mathbf{A}}$ strongly in\/ $\mathcal
E_\alpha^+(\mathbf{A}_{m})$, i.e.
\begin{equation*}\label{sigma1}\lim_{\ell\to\infty}\,\|\boldsymbol{\lambda}^{\boldsymbol{\sigma}_{\ell}}_{\mathbf{A}_{\ell}}-
\boldsymbol{\lambda}^{\boldsymbol{\sigma}}_{\mathbf{A}}\|_{\mathcal
E_\alpha^+(\mathbf{A}_{m})}=0,
\end{equation*}
and hence also vaguely provided that all the\/ $A_i$, $i\in I$, are mutually essentially disjoint.
\end{theorem}

\begin{proof} From the monotonicity of $\mathbf{A}_\ell$ and $\boldsymbol{\sigma}_{\ell}$ we get
\begin{equation}\label{eq:mon}\mathcal E^{\boldsymbol\sigma}_{\alpha,\mathbf f}(\mathbf A,\mathbf a,\mathbf g)\subset\mathcal E^{\boldsymbol\sigma_{\ell+1}}_{\alpha,\mathbf f}(\mathbf A_{\ell+1},\mathbf a,\mathbf g)\subset\mathcal E^{\boldsymbol\sigma_{\ell}}_{\alpha,\mathbf f}(\mathbf A_{\ell},\mathbf a,\mathbf g),\quad\ell\in\mathbb N,\end{equation}
and therefore
\begin{equation}\label{Gfund}-\infty<G^{\boldsymbol{\sigma}_{\ell_0}}_{\alpha,\mathbf{f}}(\mathbf{A}_{\ell_0},\mathbf{a},\mathbf{g})\leqslant\lim_{\ell\to\infty}\,G^{\boldsymbol{\sigma}_\ell}_{\alpha,\mathbf{f}}(\mathbf{A}_\ell,\mathbf{a},\mathbf{g})\leqslant
G^{\boldsymbol{\sigma}}_{\alpha,\mathbf{f}}(\mathbf{A},\mathbf{a},\mathbf{g})
<\infty,\end{equation} where the first inequality is valid by (\ref{-finite}) with $\mathbf{A}_{\ell_0}$ and $\boldsymbol{\sigma}_{\ell_0}$ in place of $\mathbf A$ and $\boldsymbol\sigma$, respectively, while the last one holds by the (standing) assumption (\ref{gauss-finite}).

According to Theorem~\ref{th-main-comp} and Theorem~\ref{th-main}, under the stated hypotheses for every $\ell\geqslant\ell_0$ there exists a minimizer
$\boldsymbol{\lambda}_\ell:=\boldsymbol{\lambda}^{\boldsymbol{\sigma}_\ell}_{\mathbf{A}_\ell}\in\mathfrak
S^{\boldsymbol{\sigma}_\ell}_{\alpha,\mathbf{f}}(\mathbf{A}_\ell,\mathbf{a},\mathbf{g})$. By (\ref{Gfund}), $\lim_{\ell\to\infty}\,G_{\alpha,\mathbf{f}}(\boldsymbol{\lambda}_\ell)$
exists and
\begin{equation}\label{Gfund2}
-\infty<\lim_{\ell\to\infty}\,G_{\alpha,\mathbf{f}}(\boldsymbol{\lambda}_\ell)\leqslant
G^{\boldsymbol{\sigma}}_{\alpha,\mathbf{f}}(\mathbf{A},\mathbf{a},\mathbf{g})
<\infty.\end{equation}
For an arbitrary fixed $m\geqslant\ell_0$ we also see from (\ref{eq:mon}) that
\[\boldsymbol{\lambda}_\ell\in\mathcal E^{\boldsymbol\sigma_m}_{\alpha,\mathbf f}(\mathbf A_m,\mathbf a,\mathbf g)\quad\text{for all  \ }\ell\geqslant m.\]

We next proceed by showing that
\begin{equation}\label{convex}
\|\boldsymbol{\lambda}_{\ell_2}-
\boldsymbol{\lambda}_{\ell_1}\|^2_{\mathcal E^+_\alpha(\mathbf A_{m})}\leqslant
G_{\alpha,\mathbf{f}}(\boldsymbol{\lambda}_{\ell_2})-G_{\alpha,\mathbf{f}}(\boldsymbol{\lambda}_{\ell_1})\quad\text{whenever
 \ }m\leqslant\ell_1\leqslant\ell_2.
\end{equation}
For any $\tau\in(0,1]$ we have $\boldsymbol{\mu}:=(1-\tau)\boldsymbol{\lambda}_{\ell_1}+
\tau\boldsymbol{\lambda}_{\ell_2}\in\mathcal
E^{\boldsymbol{\sigma}_{\ell_1}}_{\alpha,\mathbf{f}}(\mathbf{A}_{\ell_1},\mathbf{a},\mathbf{g})$, hence
$G_{\alpha,\mathbf{f}}(\boldsymbol{\mu})\geqslant
G_{\alpha,\mathbf{f}}(\boldsymbol{\lambda}_{\ell_1})$. Evaluating $G_{\alpha,\mathbf{f}}(\boldsymbol{\mu})$ and then letting $\tau\to0$, we get
\[-\kappa_\alpha(\boldsymbol{\lambda}_{\ell_1},\boldsymbol{\lambda}_{\ell_1})+
\kappa_\alpha(\boldsymbol{\lambda}_{\ell_1},\boldsymbol{\lambda}_{\ell_2})-
\langle\mathbf{f},\boldsymbol{\lambda}_{\ell_1}\rangle+
\langle\mathbf{f},\boldsymbol{\lambda}_{\ell_2}\rangle\geqslant0,\]
and (\ref{convex}) follows. Noting that, by (\ref{Gfund2}), the sequence $G_{\alpha,\mathbf{f}}(\boldsymbol{\lambda}_\ell)$, $\ell\geqslant\ell_0$, is Cauchy
in $\mathbb R$, we see from (\ref{convex}) that $\{\boldsymbol{\lambda}_{\ell}\}_{\ell\geqslant m}$
is strong Cauchy in $\mathcal E^{\boldsymbol\sigma_m}_\alpha(\mathbf{A}_m,\mathbf{a},\mathbf{g})$.

According to Lemma~\ref{lemma-rel-cl} and Remark~\ref{rem-net}, the set $\mathfrak M^{\boldsymbol\sigma_m}(\mathbf A_m,\mathbf a,\mathbf g)$ is sequentially vaguely compact. Hence there is a (strong Cauchy) subsequence $\{\boldsymbol{\lambda}_{\ell_k}\}$ of $\{\boldsymbol{\lambda}_{\ell}\}_{\ell\geqslant m}$ such that \begin{equation}\label{vaguelast}\boldsymbol{\lambda}_{\ell_k}\to\boldsymbol\lambda\text{ \ vaguely as $k\to\infty$},\end{equation}
where $\boldsymbol\lambda\in\mathfrak M^{\boldsymbol\sigma_m}(\mathbf A_m,\mathbf a,\mathbf g)$. Since the vague limit is unique, $\lambda^i$ is carried by $A^m_i$ for every $m\geqslant\ell_0$, and hence by $A_i=\bigcap_{m\geqslant\ell_0}\,A^m_i$. As the vague limit of the (positive) measures $\sigma_{\ell_k}^i-\lambda_{\ell_k}^i$ is likewise the positive measure $\sigma^i-\lambda^i$ (see\ (\ref{eq:vague}) and (\ref{vaguelast})), we altogether get
\begin{equation}\label{eq:incl}\boldsymbol{\lambda}\in\mathfrak M^{\boldsymbol\sigma}(\mathbf A,\mathbf a,\mathbf g).\end{equation}

Assume first that $\mathbf A_{\ell_0}$ and $\boldsymbol\sigma_{\ell_0}$ satisfy the assumptions of Theorem~\ref{th-main-comp}. Then so do $\mathbf A_{\ell}$ and $\boldsymbol\sigma_{\ell}$ for every $\ell\geqslant\ell_0$ and hence, according to Lemma~\ref{cont-pot}, all the potentials $\kappa_\alpha(x,\lambda_\ell^i)$ with $\ell\geqslant\ell_0$ and $i\in I$ are (bounded and finitely) continuous on the compact sets $A_\ell^i$. Applying arguments similar to those that have been applied in the proof of Theorem~\ref{th-main-comp} we then conclude from (\ref{vaguelast}) and (\ref{eq:incl}) that $\boldsymbol\lambda\in\mathcal E_\alpha^{\boldsymbol\sigma}(\mathbf A,\mathbf a,\mathbf g)$ and moreover
\begin{equation}\label{stronglast}\boldsymbol{\lambda}_{\ell_k}\to\boldsymbol\lambda\text{ \ strongly in $\mathcal
E_\alpha^+(\mathbf{A}_{m})$ as $k\to\infty$}.\end{equation}
In view of Lemma~\ref{gauss:lsc} applied to $\mathbf A_m$ instead of $\mathbf A$, we get from (\ref{Gfund2}), (\ref{vaguelast}), and (\ref{stronglast})
\[-\infty<G_{\alpha,\mathbf f}(\boldsymbol{\lambda})\leqslant\lim_{k\to\infty}\,G_{\alpha,\mathbf f}(\boldsymbol{\lambda}_{\ell_k})=\lim_{k\to\infty}\,G^{\boldsymbol\sigma_{\ell_k}}_{\alpha,\mathbf{f}}(\mathbf{A}_{\ell_k},\mathbf{a},\mathbf{g})\leqslant
G^{\boldsymbol{\sigma}}_{\alpha,\mathbf{f}}(\mathbf{A},\mathbf{a},\mathbf{g})
<\infty,\]
the first inequality being valid by (\ref{GII}) and (\ref{GI}). Thus $\boldsymbol\lambda\in\mathcal E^{\boldsymbol\sigma}_{\alpha,\mathbf f}(\mathbf A,\mathbf a,\mathbf g)$ (see (\ref{eq:incl})), and hence $G_{\alpha,\mathbf f}(\boldsymbol\lambda)\geqslant
G^{\boldsymbol\sigma}_{\alpha,\mathbf f}(\mathbf A,\mathbf a,\mathbf g)$. Combined with the last display, this  proves
(\ref{sigma}) and also
\begin{equation}\label{inclast}\boldsymbol\lambda\in\mathfrak
S^{\boldsymbol\sigma}_{\alpha,\mathbf{f}}(\mathbf{A},\mathbf{a},\mathbf{g}).\end{equation}

Assume now instead that $\mathbf A_{\ell_0}$ and $\boldsymbol\sigma_{\ell_0}$ satisfy the assumptions of Theorem~\ref{th-main}. Applying Theorem~\ref{complete}, we infer from what has been obtained above (see (\ref{vaguelast}) and (\ref{eq:incl})) that $\boldsymbol\lambda\in\mathcal E^{\boldsymbol\sigma}_\alpha(\mathbf A,\mathbf a,\mathbf g)$ and  $\boldsymbol\lambda_{\ell_k}\to\boldsymbol\lambda$ (vaguely and) strongly in $\mathcal
E_\alpha^+(\mathbf{A}_{m})$. Then in the same way as it has been established just above, we again get (\ref{sigma}) and (\ref{inclast}).

It has thus been shown that, under the hypotheses of Theorem~\ref{th:cont}, relation (\ref{sigma}) holds and $\bigl\{\boldsymbol\lambda^{\boldsymbol\sigma_\ell}_{\mathbf A_\ell}\bigr\}_{\ell\geqslant m}$, being strong Cauchy in $\mathcal E^+_\alpha(\mathbf A_m)$,
converges strongly in $\mathcal
E_\alpha^+(\mathbf{A}_{m})$ to any of its vague cluster points $\boldsymbol\lambda$, and also that this $\boldsymbol\lambda$ solves Problem~\ref{pr2} for the condenser $\mathbf A$ and the constraint $\boldsymbol\sigma$. In the case where all the $A_i$, $i\in I$, are mutually essentially disjoint, such a solution is determined uniquely according to Lemma~\ref{lemma:unique:}, so that the vague cluster set of $\bigl\{\boldsymbol\lambda^{\boldsymbol\sigma_\ell}_{\mathbf A_{\ell}}\bigr\}_{\ell\geqslant m}$ reduces to the given $\boldsymbol\lambda$. Hence $\boldsymbol\lambda^{\boldsymbol\sigma_\ell}_{\mathbf A_\ell}\to\boldsymbol\lambda$ also vaguely \cite[Chapter~I, Section~9, n$^\circ$\,1, Corollary]{B1}.
\end{proof}

\section{The $\mathbf f$-weighted vector potential of a minimizer $\boldsymbol\lambda\in\mathfrak S^{\boldsymbol{\sigma}}_{\alpha,\mathbf f}(\mathbf A,\mathbf a,\mathbf g)$}\label{sec:descr}

Theorem~\ref{desc-pot} below establishes a description to the $\mathbf f$-weighted $\alpha$-Riesz vector potentials $\mathbf W_{\alpha,\mathbf f}^{\boldsymbol\lambda}=\bigl(W_{\alpha,\mathbf f}^{\boldsymbol\lambda,i}\bigr)_{i\in I}$ (see (\ref{wpot})) of the solutions to Problem~\ref{pr2} (provided they exist), and it also singles out their characteristic properties.

\begin{lemma}\label{aux43}For\/ $\boldsymbol\lambda\in\mathcal E^{\boldsymbol\sigma}_{\alpha,\mathbf f}(\mathbf A,\mathbf a,\mathbf g)$ to solve Problem\/~{\rm\ref{pr2}}, it is necessary and sufficient that
\begin{equation}\label{aux431}\sum_{i\in I}\,\bigl\langle W^{\boldsymbol\lambda,i}_{\alpha,\mathbf f},\nu^i-\lambda^i\bigr\rangle\geqslant0\quad\text{for all \ }\boldsymbol\nu\in\mathcal E^{\boldsymbol\sigma}_{\alpha,\mathbf f}(\mathbf A,\mathbf a,\mathbf g).\end{equation}
\end{lemma}

\begin{proof}By direct calculation, for any $\boldsymbol\mu,\boldsymbol\nu\in\mathcal E^{\boldsymbol\sigma}_{\alpha,\mathbf f}(\mathbf A,\mathbf a,\mathbf g)$ and any $h\in(0,1]$ we obtain
\[G_{\alpha,\mathbf f}\bigl(h\boldsymbol\nu+(1-h)\boldsymbol\mu\bigr)-G_{\alpha,\mathbf f}(\boldsymbol\mu)=2h\sum_{i\in I}\,\bigl\langle W^{\boldsymbol\mu,i}_{\alpha,\mathbf f},\nu^i-\mu^i\bigr\rangle+h^2\|\boldsymbol\nu-\boldsymbol\mu\|^2_{\mathcal E^+_\alpha(\mathbf A)}.\]
If $\boldsymbol\mu=\boldsymbol\lambda$ solves Problem~\ref{pr2} then the left hand (and hence the right hand) side of this display is ${}\geqslant0$, for the class
$\mathcal E^{\boldsymbol\sigma}_{\alpha,\mathbf f}(\mathbf A,\mathbf a,\mathbf g)$ is convex, which leads to (\ref{aux431}) by letting $h\to0$ (after division by $h$).
Conversely, if (\ref{aux431}) holds, then the preceding formula with $\boldsymbol\mu=\boldsymbol\lambda$ and $h=1$ implies that $G_{\alpha,\mathbf f}(\boldsymbol\nu)\geqslant G_{\alpha,\mathbf f}(\boldsymbol\lambda)$ for all $\boldsymbol\nu\in\mathcal E^{\boldsymbol\sigma}_{\alpha,\mathbf f}(\mathbf A,\mathbf a,\mathbf g)$, hence $\boldsymbol\lambda\in\mathfrak S^{\boldsymbol{\sigma}}_{\alpha,\mathbf f}(\mathbf A,\mathbf a,\mathbf g)$.
\end{proof}

\begin{theorem}\label{desc-pot} Assume that for every\/ $i\in I$, $\kappa_\alpha(\cdot,\sigma^i)$ is {\rm(}finitely\/{\rm)} continuous on\/ $A_i$ and upper bounded on some neighborhood of\/ $\omega_{\mathbb R^n}$, and
\begin{equation}\label{as1}\sigma^i(A_i\setminus\dot{A}_i^\delta)=0,\end{equation}
$\dot{A}_i^\delta$ being defined by\/ {\rm(\ref{circ})}. If moreover Case\/~{\rm I} takes place and
\begin{equation}\label{as2}g_i(x)\leqslant M_i<\infty\text{ \ for all \ }x\in\mathbb R^n, \ i\in I,\end{equation}  then for any given\/ $\boldsymbol\lambda\in\mathcal E^{\boldsymbol\sigma}_{\alpha,\mathbf f}(\mathbf A,\mathbf a,\mathbf g)$ the following two assertions are equivalent:
\begin{itemize}
\item[(i)] $\boldsymbol\lambda\in\mathfrak S^{\boldsymbol{\sigma}}_{\alpha,\mathbf f}(\mathbf A,\mathbf a,\mathbf g)$.
\item[(ii)] There exists\/ $(w_{\boldsymbol\lambda}^i)_{i\in I}\in\mathbb R^{|I|}$ {\rm(}where\/ $|I|:={\rm Card}\,I${\rm)} such that for all\/ $i\in I$
\begin{align}\label{b1}W^{\boldsymbol\lambda,i}_{\alpha,\mathbf f}&\geqslant w_{\boldsymbol\lambda}^ig_i\quad(\sigma^i-\lambda^i)\text{-a.e.\ on \ }A_i,\\
\label{b2}W^{\boldsymbol\lambda,i}_{\alpha,\mathbf f}&\leqslant w_{\boldsymbol\lambda}^ig_i\quad\text{everywhere on \ }S_{\mathbb R^n}^{\lambda^i}.
\end{align}
\end{itemize}
\end{theorem}

\begin{proof} Since for every $i\in I$, $\kappa_\alpha(\cdot,\sigma^i)$ is continuous on $A_i$, so is $\kappa_\alpha(\cdot,\mu^i)$, where $(\mu^i)_{i\in I}$ is any measure from $\mathfrak M^{\boldsymbol\sigma}(\mathbf A,\mathbf a,\mathbf g)$ (see Lemma~\ref{cont-pot}).
By \cite[Theorem~1.7]{L}, all these potentials are then continuous on all of $\mathbb R^n$. As $\kappa_\alpha(\cdot,\sigma^i)$, and hence $\kappa_\alpha(\cdot,\mu^i)$ is bounded on some neighborhood of $\omega_{\mathbb R^n}$, it thus follows that all these potentials are bounded on all of $\mathbb R^n$. Another consequence is that
\begin{equation}\label{as3}\sigma^i|_K\in\mathcal E_\alpha^+(K)\text{ \ for any compact \ }K\subset\mathbb R^n,\end{equation}
and hence the measures $\sigma^i$, $i\in I$, are $c_\alpha$-absolutely continuous.

Suppose first that (i) holds, i.e.\ $\boldsymbol\lambda\in\mathcal E^{\boldsymbol\sigma}_{\alpha,\mathbf f}(\mathbf A,\mathbf a,\mathbf g)$ solves Problem~\ref{pr2}. To verify (ii), fix $i\in I$. For every $\boldsymbol\mu=(\mu^\ell)_{\ell\in I}\in\mathcal E^{\boldsymbol\sigma}_{\alpha,\mathbf f}(\mathbf A,\mathbf a,\mathbf g)$ write $\boldsymbol\mu_i:=(\mu_i^\ell)_{\ell\in I}$ where $\mu_i^\ell:=\mu^\ell$ for all $\ell\ne i$ and $\mu_i^i=0$; then $\boldsymbol\mu_i\in\mathcal E^+_{\alpha,\mathbf f}(\mathbf A)$. Also define $\tilde{f}_i:=f_i+(\kappa_\alpha)^i_{\boldsymbol\lambda_i}$; then by substituting (\ref{potv}) we obtain
\begin{equation}\label{ftilde}\tilde{f}_i(x)=f_i(x)+s_i\sum_{\ell\in I, \ \ell\ne i}\,s_\ell\kappa_\alpha(x,\lambda^\ell), \ x\in\mathbb R^n.\end{equation}
Being of the class $\Psi(\mathbb R^n)$, $f_i$ is l.s.c.\ on $\mathbb R^n$ and ${}\geqslant0$. Combined with the properties of $\kappa_\alpha(\cdot,\lambda^\ell)$, $\ell\in I$, established above, this implies that
\begin{equation}\label{sW}W_{\alpha,\tilde{f}_i}^{\lambda^i}:=\kappa_\alpha(\cdot,\lambda^i)+\tilde{f}_i,\end{equation}
is l.s.c.\ on $\mathbb R^n$ and lower bounded. Also note that $W_{\alpha,\tilde{f}_i}^{\lambda^i}$ is finite on $\dot{A}_i^\delta$ (see (\ref{circ})).

Furthermore, by (\ref{env}) and (\ref{wen}) we get for any
$\boldsymbol\mu\in\mathcal E^{\boldsymbol\sigma}_{\alpha,\mathbf f}(\mathbf A,\mathbf a,\mathbf g)$ with the additional property that $\boldsymbol\mu_i=\boldsymbol\lambda_i$ (in particular, for $\boldsymbol\mu=\boldsymbol\lambda$) \[G_{\alpha,\mathbf f}(\boldsymbol\mu)=G_{\alpha,\mathbf f}(\boldsymbol\lambda_i)+G_{\alpha,\tilde{f}_i}(\mu^i).\]
Combined with $G_{\alpha,\mathbf f}(\boldsymbol\mu)\geqslant G_{\alpha,\mathbf f}(\boldsymbol\lambda)$, this yields $G_{\alpha,\tilde{f}_i}(\mu^i)\geqslant G_{\alpha,\tilde{f}_i}(\lambda^i)$, and hence $\lambda^i$ minimizes $G_{\alpha,\tilde{f}_i}(\nu)$ where $\nu$ ranges over $\mathcal E^{\sigma^i}_{\alpha,\tilde{f}_i}(A_i,a_i,g_i)$.
This enables us to show that there exists $w_{\lambda^i}\in\mathbb R$ such that
\begin{align}\label{sing1}W_{\alpha,\tilde{f}_i}^{\lambda^i}&\geqslant w_{\lambda^i}g_i\quad(\sigma^i-\lambda^i)\text{-a.e.\ on \ }A_i,\\
\label{sing2}W_{\alpha,\tilde{f}_i}^{\lambda^i}&\leqslant w_{\lambda^i}g_i\quad\text{everywhere on \ }S^{\lambda^i}_{\mathbb R^n}.\end{align}

Indeed, (\ref{sing1}) holds with
\[w_{\lambda^i}:=L_i:=\sup\,\bigl\{t\in\mathbb R: \ W_{\alpha,\tilde{f}_i}^{\lambda^i}\geqslant tg_i\quad(\sigma^i-\lambda^i)\text{-a.e.\ on \ }A_i\bigr\}.\]
In turn, (\ref{sing1}) with $w_{\lambda^i}=L_i$ implies that $L_i<\infty$, because
\[\widetilde{W}_{\alpha,\tilde{f}_i}^{\lambda^i}(x):=\frac{W_{\alpha,\tilde{f}_i}^{\lambda^i}(x)}{g_i(x)}<\infty\] for all $x\in\dot{A}_i^\delta$ (see (\ref{infg})), hence $(\sigma^i-\lambda^i)$-a.e.\ on $A_i$ by (\ref{as1}). Also, $L_i>-\infty$ since in consequence of (\ref{as2}),
$\widetilde{W}_{\alpha,\tilde{f}_i}^{\lambda^i}$ is lower bounded on~$\mathbb R^n$.

We next proceed by establishing (\ref{sing2}) with $w_{\lambda^i}=L_i$. Assume, on the contrary, that this does not hold. For any $w\in\mathbb R$ write
\begin{align*}A_i^+(w)&:=\bigl\{x\in A_i:\ W_{\alpha,\tilde{f}_i}^{\lambda^i}(x)>wg_i(x)\bigr\},\\
A_i^-(w)&:=\bigl\{x\in A_i:\ W_{\alpha,\tilde{f}_i}^{\lambda^i}(x)<wg_i(x)\bigr\}.\end{align*}
By the lower semicontinuity of $\widetilde{W}_{\alpha,\tilde{f}_i}^{\lambda^i}$ on $\mathbb R^n$, there is $w_i\in(L_i,\infty)$ such that
$\lambda^i(A_i^+(w_i))>0$. At the same time, as $w_i>L_i$, (\ref{sing1}) with $w_{\lambda^i}=L_i$  yields
$(\sigma^i-\lambda^i)(A_i^-(w_i))>0$. Therefore, one can choose compact sets $K_1\subset A_i^+(w_i)$ and $K_2\subset A_i^-(w_i)$ so that
\begin{equation}\label{c01}0<\langle g_i,\lambda^i|_{K_1}\rangle<\langle g_i,(\sigma^i-\lambda^i)|_{K_2}\rangle.\end{equation}

Write $\tau^i:=(\sigma^i-\lambda^i)|_{K_2}$; then $\kappa_\alpha(\tau^i,\tau^i)<\infty$ by (\ref{as3}). Since $\langle W_{\alpha,\tilde{f}_i}^{\lambda^i},\tau^i\rangle\leqslant\bigl\langle w_ig_i,\tau^i\bigr\rangle<\infty$,
we get $\langle\tilde{f}_i,\tau^i\rangle<\infty$ in view of (\ref{sW}). Define
\[\theta^i:=\lambda^i-\lambda^i|_{K_1}+c_i\tau^i,\text{ \ where \ }c_i:=\langle g_i,\lambda^i|_{K_1}\rangle/\langle g_i,\tau^i\rangle.\]
Noting that $c_i\in(0,1)$ by (\ref{c01}), we see by straightforward verification that $\langle g_i,\theta^i\rangle=a_i$ and $\theta^i\leqslant\sigma^i$, hence $\theta^i\in\mathcal E^{\sigma^i}_{\alpha,\tilde{f}_i}(A_i,a_i,g_i)$. On the other hand,
\begin{align*}
\langle W_{\alpha,\tilde{f}_i}^{\lambda^i},\theta^i-\lambda^i\rangle&=\langle
W_{\alpha,\tilde{f}_i}^{\lambda^i}- w_ig_i,\theta^i-\lambda^i\rangle\\&{}=-\langle
W_{\alpha,\tilde{f}_i}^{\lambda^i}- w_ig_i,\lambda^i|_{K_1}\rangle+c_i\langle
W_{\alpha,\tilde{f}_i}^{\lambda^i}- w_ig_i,\tau^i\rangle<0,\end{align*}
which is impossible in view of the scalar version of Lemma~\ref{aux43}. The contradiction obtained establishes (\ref{sing2}).

Substituting (\ref{ftilde}) into (\ref{sW}) and then comparing the result obtained with (\ref{potv}) and (\ref{wpot}), we get
\begin{equation}\label{ww}W_{\alpha,\tilde{f}_i}^{\lambda^i}=W_{\alpha,\mathbf f}^{\boldsymbol\lambda,i}.\end{equation}
Combined with  (\ref{sing1}) and (\ref{sing2}), this proves (\ref{b1}) and
(\ref{b2}) with $w_{\boldsymbol\lambda}^i:=w_{\lambda^i}$, $i\in I$.

Conversely, suppose (ii) holds. On account of (\ref{ww}), for every $i\in I$ relations (\ref{sing1}) and (\ref{sing2}) are then fulfilled with $w_{\lambda^i}:=w_{\boldsymbol\lambda}^i$ and $\tilde{f}_i$ defined by (\ref{ftilde}). This yields
$\lambda^i(A_i^+(w_{\lambda^i}))=0$  and $(\sigma^i-\lambda^i)(A_i^-(w_{\lambda^i}))=0$.
For any $\boldsymbol\nu\in\mathcal E^{\boldsymbol\sigma}_{\alpha,\mathbf f}(\mathbf A,\mathbf a,\mathbf g)$ we therefore get
\begin{align*}\langle
W^{\boldsymbol\lambda,i}_{\alpha,\mathbf f}&,\nu^i-\lambda^i\rangle=\langle
W_{\alpha,\tilde{f}_i}^{\lambda^i}-w_{\lambda^i}g_i,\nu^i-\lambda^i\rangle\\
&{}=\langle W_{\alpha,\tilde{f}_i}^{\lambda^i}-w_{\lambda^i}g_i,\nu^i|_{A_i^+(w_{\lambda^i})}\rangle+\langle
W_{\alpha,\tilde{f}_i}^{\lambda^i}-w_{\lambda^i}g_i,(\nu^i-\sigma^i)|_{A_i^-(w_{\lambda^i})}\rangle\geqslant0.\end{align*}
Summing up these inequalities over all $i\in I$, in view of the arbitrary choice of $\boldsymbol\nu\in\mathcal E^{\boldsymbol\sigma}_{\alpha,\mathbf f}(\mathbf A,\mathbf a,\mathbf g)$ we conclude from Lemma~\ref{aux43} that $\boldsymbol\lambda$ is a solution to Problem~\ref{pr2}.\end{proof}

\section{Duality relation between non-weighted constrained and weighted unconstrained minimum $\alpha$-Riesz energy problems for scalar measures}\label{sec:dual}

We now present an extension of \cite[Corollary 2.15]{DS} to Riesz kernels. Throughout this section, set $I=\{1\}$, $s_1=+1$, $g_1=1$ and $a_1=1$. Fix a closed set $F=A_1$ in $\mathbb R^n$ with $c_\alpha(F)>0$ that may coincide with the whole of $\mathbb R^n$ and a constraint $\sigma\in\mathfrak M^+(F)$ with $1<\sigma(F)<\infty$. By Theorem~\ref{th-main} (see also \cite[Theorem~5.1]{DFHSZ} with $D=\mathbb R^n$), there exists $\lambda\in\mathcal E_\alpha^\sigma(F,1):=\mathfrak M^\sigma(F)\cap\mathcal E_\alpha^+(F,1)$ whose $\alpha$-Riesz energy is minimal in this class, i.e.
\begin{equation}\|\lambda\|_\alpha^2=\inf_{\mu\in\mathcal E_\alpha^\sigma(F,1)}\,\|\mu\|_\alpha^2,\label{pr2sc}\end{equation}
and this $\lambda$ is unique according to Lemma~\ref{lemma:unique:} (with $\mathbf f=\mathbf 0$). Write $q:=1/(\sigma(F)-1)$.

\begin{theorem}\label{dual}Assume in addition that\/ $0<\alpha\leqslant2$ and that\/ $\kappa_\alpha(\cdot,\sigma)$ is\/ {\rm(}finitely\/{\rm)} continuous and upper bounded on\/ $F$ {\rm(}hence, on\/ $\mathbb R^n$ according to\/ {\rm\cite[{\it Theorems\/}~1.5, 1.7]{L}}{\rm)}. Then the measure\/ $\theta:=q(\sigma-\lambda)$ is the solution to\/ {\rm(}the unconstrained\/{\rm)} Problem\/~{\rm\ref{pr1}} with the external field\/ $f:=-q\kappa_\alpha(\cdot,\sigma)$, i.e.\ $\theta\in\mathcal E_{\alpha,f}^+(F,1):=\mathfrak M^+(F,1)\cap\mathcal E_{\alpha,f}^+(F)$ and
\begin{equation}\label{dualeq}G_{\alpha,f}(\theta)=\inf_{\mu\in\mathcal E_{\alpha,f}^+(F,1)}\,G_{\alpha,f}(\mu).\end{equation}
Moreover, there exists\/ $\eta\in(0,\infty)$ such that
\begin{align}
\label{Wsc1}W_{\alpha,f}^\theta&=-\eta\text{ \ on \ }S(\theta),\\
\label{Wsc2}W_{\alpha,f}^\theta&\geqslant-\eta\text{ \ on \ }\mathbb R^n,
\end{align}
and these two relations\/ {\rm(\ref{Wsc1})} and\/ {\rm(\ref{Wsc2})} determine uniquely the solution to Problem\/~{\rm\ref{pr1}} among the measures of the class\/ $\mathcal E_{\alpha,f}^+(F,1)$.
\end{theorem}

\begin{remark}The external field $f$ thus defined satisfies Case~II with $\zeta=-q\sigma\leqslant0$, and this $f$ is lower bounded on $\mathbb R^n$, for $\kappa_\alpha(\cdot,\sigma)$ is upper bounded by assumption.\end{remark}

\begin{proof} Under the stated assumptions, relation (\ref{b2}) for the solution $\lambda$ to the (constrained) problem (\ref{pr2sc}) takes the form $\kappa_\alpha(\cdot,\lambda)\leqslant w$ on $S(\lambda)$, where $w\in(0,\infty)$. By the Frostman maximum principle, which can be applied because $\alpha\leqslant2$, we thus have
\begin{equation*}\kappa_\alpha(\cdot,\lambda)\leqslant w\text{ \ on \ }\mathbb R^n.\end{equation*}
Combined with (\ref{b1}), this gives
\begin{equation*}\kappa_\alpha(\cdot,\lambda)=w\quad\text{on \ }S(\sigma-\lambda),\end{equation*}
for $\kappa_\alpha(\cdot,\lambda)$ is (finitely) continuous on $\mathbb R^n$ along with $\kappa_\alpha(\cdot,\sigma)$ by Lemma~\ref{cont-pot}. In the notations used in Theorem~\ref{dual}, these two displays can alternatively be rewritten as (\ref{Wsc2}) and (\ref{Wsc1}), respectively, with $\eta:=qw$. In turn, (\ref{Wsc1}) and (\ref{Wsc2}) imply that $\theta$, $f$ and $-\eta$ satisfy \cite[Eqs.~7.9, 7.10]{ZPot2}, which according to \cite[Theorem~7.3]{ZPot2} establishes (\ref{dualeq}).
\end{proof}

{\bf Acknowledgement.} The authors express their sincere gratitude to the anonymous Referee for valuable suggestions, helping us in improving the exposition of the paper.

\end{document}